\documentclass[final,1p,times]{elsarticle}
\usepackage{amssymb}
\usepackage{amsmath}
\usepackage{amsthm}
\usepackage{cleveref}
\usepackage{algorithm}
\usepackage{algorithmic}
\usepackage{color}
\newcommand{\R}{\mathbb{R}}
\newcommand{\N}{\mathbb{N}}
\newtheorem{thm}{Theorem}

 \newtheorem{prop}[thm]{Proposition}

\newtheorem{rem}[thm]{Remark}

\journal{Journal of theoretical biology}

\begin{document}

\begin{frontmatter}



\title{On a Keller-Segel type equation to model Brain Microvascular Endothelial Cells growth's patterns}

\author[label1,label2]{B. Ambrosio}
 \affiliation[label1]{organization={ University Le Havre Normandie, Normandie Univ.,
LMAH UR 3821},city={Le Havre},
             postcode={76600},
             country={France}}
 \affiliation[label2]{organization={The Hudson School of Mathematics},
             addressline={244 Fifth Avenue, Suite Q224 },
             city={ New York},
             postcode={10001},
             state={NY},
             country={USA}}
             
\author[label3]{A.F. Garroudji}
 \author[label4,label5]{S. Fitzsimons}
\author[label3]{H. Zaag}
\author[label5]{F.M. Elahi} %% Author name

\affiliation[label3]{organization={University Sorbonne Paris Nord, LAGA, CNRS (UMR 7539)},
             city={Villetaneuse},
             postcode={93430},
             country={France}
             }

              \affiliation[label4]{organization={School of Medicine, Conway Institute, University College of Dublin}, 
               city={Belfield, Dublin},
             postcode={4},
             country={Ireland}}
 \affiliation[label5]{organization={Departments of Neurology and Neuroscience, Ronald M. Loeb Center for Alzheimer’s Disease, Friedman Brain Institute, Glickenhaus Center for Successful Aging, Icahn School of Medicine at Mount Sinai},
             addressline={1 
Gustave L. Levy Place},
             city={New York},
             postcode={10029-5674},
             state={NY},
             country={USA}}


%% Abstract
\begin{abstract}
This article introduces a Keller–Segel-type partial differential equation (PDE) model that reproduces patterns commonly observed during the growth of cerebral microvasculature. We provide mathematical insight into the mechanisms driving the emergence of these patterns and derive a reduced system of two coupled two-dimensional ordinary differential equations that captures key aspects of the underlying dynamics. We also develop a data-driven equation that yields a chemoattractant evolution consistent with the observed temporal dynamics of the microvasculature. Beyond numerical simulations, this study contributes to the development of a comprehensive mathematical modeling framework that integrates processes across multiple scales—from blood flow in cerebral arterial networks to biochemical interactions—with the broader goal of clarifying how vascular dysfunction may contribute to the onset and progression of neurodegenerative diseases.
\end{abstract}



%% Keywords

\end{frontmatter}




\section{Introduction}
Neurodegenerative diseases (ND) constitute a major and growing global health burden. It is estimated that more than 3 billion people currently live with neurological conditions \cite{Ste2024}, while dementia alone is projected to affect approximately 152 million individuals by 2050 \cite{Nich2022}. Increasing evidence indicates that dysregulation of cerebral blood flow plays an important role in the development and progression of several forms of dementia \cite{Win2024,Owe2024,Tor2024,Zhu2022}. Consequently, understanding the structure and dynamics of the brain microvasculature in both healthy and pathological conditions is of central importance.

Recent advances in stem-cell technologies provide powerful tools to investigate the mechanisms of brain microvascular growth. In particular, induced pluripotent stem cells (iPSCs) can be differentiated into brain microvascular endothelial cells (BMECs). BMECs can also be directly isolated from the human brain, enabling the construction of multicellular \textit{in vitro} models of the neuro–glio–vascular unit. These models offer a valuable framework for identifying the molecular drivers of pathological cell–cell interactions within this system \cite{Elahi2023}.

The formation of vascular networks \textit{ in vitro} has been observed and studied for several decades (see, for example, \cite{Good2007}). More recently, the increasing availability of large-scale biological datasets and the development of computational analysis tools have created new opportunities to identify molecular markers associated with vascular pattern formation and to investigate correlations between spatial cellular patterns and omics data in the brain.

From the perspective of applied mathematics, pattern formation has long been studied through the framework of reaction–diffusion (RD) systems. These models are known to generate a wide variety of spatial and spatio-temporal structures through mechanisms such as Turing instabilities, Hopf bifurcations, phase-locked oscillations, and chaotic dynamics; see, for example, \cite{Murray,Turing,Amb-2012,Amb-2016}. In the context of cellular aggregation and pattern formation, a classical model is the Keller–Segel (KS) system \cite{Kel1971}, which incorporates both diffusive transport and chemotactic cell migration.

In this work, we study a Keller–Segel-type model capable of reproducing spatial patterns similar to those observed in BMEC cultures \textit{ in vitro}. We first present the model and provide a preliminary analysis of the associated ordinary differential equation (ODE). Numerical simulations are then used to illustrate how variations in key parameters lead to the emergence of network-like spatial structures resembling experimentally observed BMEC patterns. We further propose a data-driven extension of the model that ensures a consistent temporal evolution of the chemo-attractant associated with the observed BMEC dynamics. Finally, we analyze the dynamical behavior of the model in greater detail and introduce a reduced four-dimensional system of $2\times2$ coupled ODEs that captures essential features of the underlying dynamics.

The experimental data used for the mathematical modeling performed here were generated by confocal microscopy of microvasculature composed of  BMECs and stem cell derived mural cells, in the Elahi Lab at the Icahn School of Medicine at Mount Sinai. BMECs were isolated from human brain (obtained from ScienCell Research Laboratories, Carlsbad, CA, USA; Cat. No. 1000). Mural Cells were differentiated from iPSC lines according to well-established protocols \cite{Blanchard2020APOE4,Lippmann2012BBB,Stanton2025miBrain}. Cells were seeded in 96-well plates coated with Geltrex and cultured in Endothelial Cell Medium supplemented with vascular endothelial growth factor (VEGF, 5 ng /mL) and fibroblast growth factor 2 (FGF2, 2 ng/mL).

\section{On a modified Keller-Segel Equation}
\subsection{Keller-Segel Equations}
Relying on observations of \textit{Escherichia coli} placed in an environment containing oxygen and an energy source, Keller and Segel \cite{Kel1971} proposed a partial differential equation (PDE) to model the formation of traveling bands in bacterial concentrations observed on plates. The central idea of their work was to interpret this phenomenon as a consequence of chemotaxis: bacteria tend to avoid regions of low concentration and move preferentially toward higher concentrations of the substrate. Additional effects incorporated into the model include random motion (diffusion) and substrate consumption, while growth was not taken into account.

The original Keller--Segel \cite{Kel1971} equation is given by
\begin{equation}
\label{eq:KS}
    \left\{ \begin{array}{rl}
         \frac{\partial b}{ \partial t}&=\frac{\partial }{ \partial x}\big(\mu(s) \frac{\partial b}{ \partial x} \big) -\frac{\partial }{ \partial x}\big(b\xi(s) \frac{\partial s}{ \partial x} \big)  \\
           \frac{\partial s}{ \partial t}&=-k(s)b+D \frac{\partial s}{ \partial x^2}  \\
    \end{array}
    \right.
\end{equation}

These notations follow those of the original article. Here, $b$ denotes the bacterial concentration and $s$ the substrate concentration. The first term on the right-hand side of the first equation represents bacterial motion in the absence of chemotaxis. The second term on the right-hand side of the first equation describes the chemotactic response of the bacteria. It is assumed that the portion of the bacterial flux resulting from chemotaxis is proportional to the chemical gradient. This assumption is analogous to those used in the derivation of the heat equation. 

Since then, numerous variants of the model have been studied both numerically and theoretically. For example, a growth term was introduced and discussed in \cite{Mim1996}. We refer to \cite{Mim1996,Dol2024,Cor2004,Kut2012} and the references therein for relevant mathematical and numerical studies. The goal of this paper is to propose and analyze a KS-type equation that reproduces relevant patterns for BMEC networks, and to further highlight research directions that appear important for improving our understanding of ND.
\subsection{On a modified Keller-Segel Equation}
We consider the following modified KS type equation,
\begin{equation}
\left\{
    \begin{array}{rl}
         u_t&=f(u)-b\nabla \cdot (u\nabla v)+d_u\Delta u  \\
         v_t&=cu-ev+d_v\Delta v 
    \end{array}
    \right.
    \label{eq:mKS}
\end{equation}
on a regular bounded domain $\Omega\subset \R^n, n=2 \mbox{ or } 3$, with Neuman Boundary Conditions (NBC), and with 
\[f(u)=au(1-u)(u-\gamma).\]

The notations are as follows: the  term $\nabla \cdot $ stands for the divergence operator, the  term $\nabla$ for the gradient and the term $\Delta$ for the Laplacian. The variable $u$ represents the concentration of endothelial cells and $v$ represents the substrate and the biochemical mechanisms through which cells communicate. This equation is of Keller-Segel type with a modified reaction function. Let us briefly recall the fundamental principles underlying \eqref{eq:mKS}. The first equation stipulates that the concentration of cells in an infinitesimal volume will evolve as a result of an intrinsic production function $f$, the divergence of the concentration of $-bu$ times the gradient of $v$ (which means that if there is more $v$ outside, the concentration at the current point will decrease; this is explained by the divergence theorem), and a diffusion in $u$, representing the fact that the concentration evolves like the average of its surroundings. As for $v$, it is produced from $u$ with a rate $c$ and it naturally disappears at a rate $e$. We introduce the modified function $f$ so that the underlying ODE $u'=f(u)$ would have three stationary states: two stable $0$ and $1$, and one unstable $\gamma$. We think that the function $f$ is relevant to model patterns formed by BMEC, because after some time it is observed in experiments that one can roughly divide the two-dimensional well in regions where there is BMEC ($u\simeq 1$) and regions where there is no BMEC ($u\simeq 0$). Investigating the dynamical behavior with variation of the parameter $\gamma$ leads to various configurations, including those observed in experiments. The numerical simulations are presented in \cref{sec:NSAP}. We note that \Cref{eq:mKS} equivalently writes
\begin{equation}
\left\{
    \begin{array}{rl}
         u_t&=f(u)-b\nabla u \cdot \nabla v-bu\Delta v+d_u\Delta u  \\
         v_t&=cu-ev+d_v\Delta v 
    \end{array}
    \right.
    \label{eq:mKSbis}
\end{equation}
This last formulation will be used in section \ref{sec:NS3timephase} for a qualitative description.
\subsection{Basic Mathematical Analysis of \cref{eq:mKS}} 
\subsubsection{The underlying ODE}
We consider,
\begin{equation}
\left\{
    \begin{array}{rl}
         u_t&=f(u)  \\
         v_t&=cu-ev
    \end{array}
    \right.
    \label{eq:mKSODE}
\end{equation}
We have the following result, which follows from straightforward computations,
\begin{prop}
    \Cref{eq:mKSODE} admits three stationary points $U^*_0=(0,0)$,$U^*_1=(\gamma,\frac{c}{e}\gamma)$, and $U^*_2=(1,\frac{c}{e})$. $U^*_0$ and $U^*_2$ are stable nodes, $U^*_1$ is a saddle-node. The stable manifold of $U^*_1$ is given by $u=\gamma$. If the initial condition $u_0$ satisfies $0\leq u_0<\gamma$ then the trajectory converges toward $U^*_0$, if   $u_0>\gamma$ the trajectory converges toward $U^*_2$.
\end{prop}
\Cref{fig:1} provides an illustration of the solutions of \Cref{eq:mKSODE}.
\begin{figure}
    \centering
    \includegraphics[height=5cm,width=5cm]{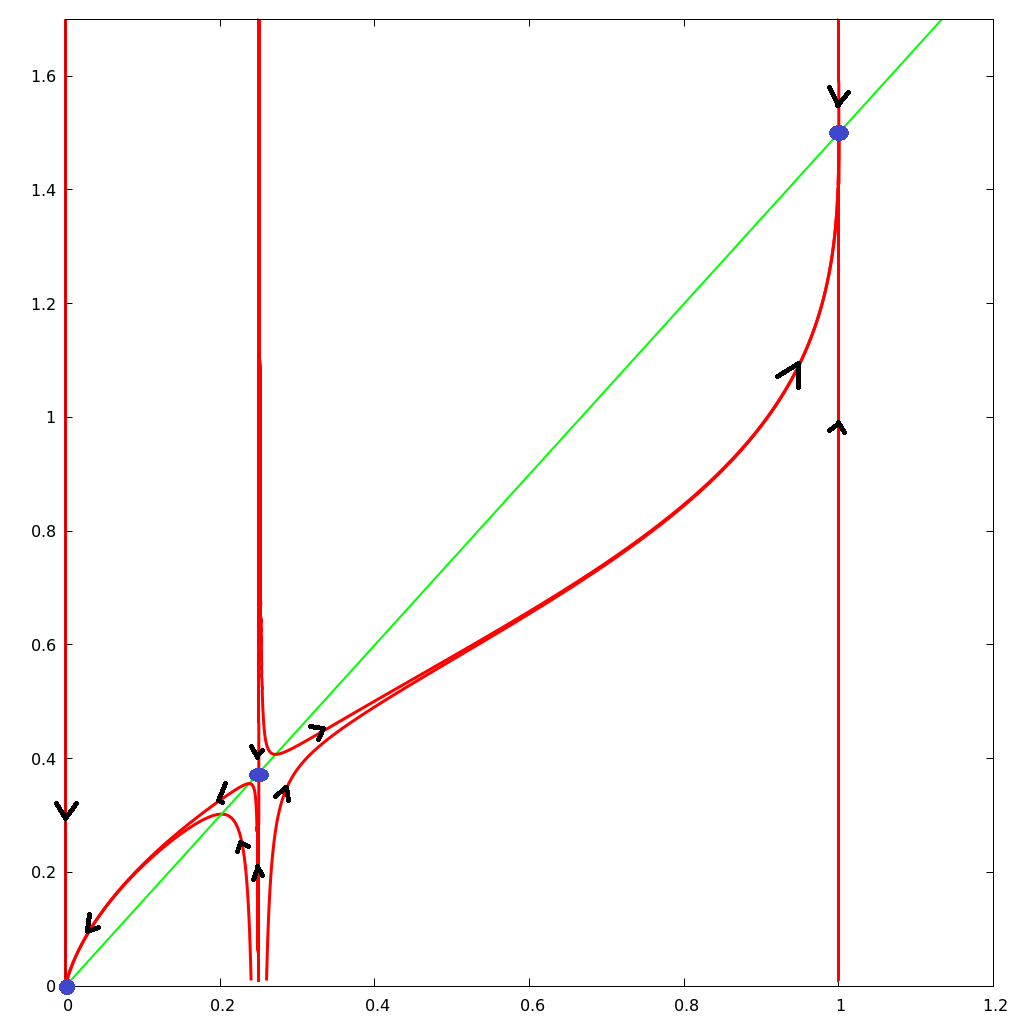}
    \caption{Solutions of  the ODE \cref{eq:mKSODE}}
    \label{fig:1}
\end{figure}
\subsubsection{Positivity, Mean value}
In this section, we provide two basic mathematical results. The first one deals with the positivity of solutions which must be verified in order to satisfy the biological meaning of the variables. The second one deals with the average values. We assume here the existence and uniqueness of classical solutions (see \cite{Ama-2026-prep})
\begin{prop}
Assume that the initial data $u_0,v_0$ are smooth and non-negative. Then the solution $(u,v)$ of \Cref{eq:mKSbis} remains non-negative, i.e.,
\[
u(x,t)\geq 0,\qquad v(x,t)\geq 0
\]
for all $(x,t)\in\overline{\Omega}\times[0,+\infty)$.
\end{prop}
\begin{proof}

We omit the details of the proof, which follow from standard maximum
principle arguments; see, for example, the classical textbooks
\cite{BookProtter-Weinberger,BookSmoller}. 
\end{proof}
\begin{rem}
We remark here that the solutions remain non-negative, $u\geq 0, v\geq 0$, with $0$ corresponding to the lower stable stationary solution of the bistable equation $u'=f(u)$.  A key idea of the proof behind the maximum principle revolves around the following reasoning.  Let's assume that $u$ first reaches a minimal value $u_m<0$ at an interior point $(t_0,x_0)$ with $v$ positive then, in the first equation 
\[u_t\leq 0\]
but on the other hand,
\[
f(u_m)>0, \nabla u= 0,u\Delta v=0, \Delta u \geq 0.
\]
which provides
\[
u_t=f(u)-b \nabla u\cdot \nabla v -b u\Delta v +d_u\Delta u> 0.
\]
which is a contradiction. It is natural to ask whether the same argument holds at $u=1$ which corresponds to the other stable stationary point of the bistable equation $u'=f(u)$. Reasoning analogously, we assume that $u$ first reaches a maximal value $u_M>1$ at an interior point $(t_0,x_0)$ with $v$ positive then, in the first equation 
\[u_t\geq 0\]
and on the other hand,
\[
f(u_M)<0, \nabla u= 0, \Delta u \leq 0.
\]
but one may have 
\[u\Delta v>0.\]
So because of the term $u\Delta v$, one cannot prove that $u\leq 1$.
\end{rem}
The average concentration values are of practical interest since they can serve as a measure of degeneracy for diseased cells. They satisfy specific equations, as the following proposition states.
Let 
\[\bar{u}(t)=\int_{\Omega}u(x,t)dx,\, and \,\bar{v}(t)=\int_{\Omega}v(x,t)dx,  \]
then the following results hold. 
\begin{prop}
$\bar{u}$ and  $\bar{v}$ satisfy
    \begin{equation}
\left\{
    \begin{array}{rl}
         \bar{u}_t&=\int_\Omega f(u)dx  \\
         \bar{v}_t&=c\bar{u}-e\bar{v}
    \end{array}
    \right.
    \label{eq:barODE}
\end{equation}
which for stationary solutions of \Cref{eq:mKS} implies
 \begin{equation}
\left\{
    \begin{array}{rl}
         0&=\int_\Omega f(u)dx  \\
         0&=c\bar{u}-e\bar{v}
    \end{array}
    \right.
    \label{eq:StaBarODE}
\end{equation}
\end{prop}
\subsubsection{Stability of the Stationary Homogeneous Constant Solutions and Bifurcation threshold}
\textbf{Linearization}
Here, we consider the linearization of \Cref{eq:mKSbis} around the stationary solutions $(u,v)(x,t)=(0,0)$ and $(u,v)(x,t)=(1,\frac{c}{e})$.
Around $(0,0)$ the linearized system reads
\begin{equation}
\left\{
    \begin{array}{rl}
         u_t&=-a\gamma u+d_u\Delta u,  \\
         v_t&=cu-ev+d_v\Delta v. 
    \end{array}
    \right.
    \label{eq:lin_at_0}
\end{equation}
Around $(1,\frac{c}{e})$, it reads
\begin{equation}
\left\{
    \begin{array}{rl}
         u_t&=a(\gamma-1)u-b\Delta v +d_u\Delta u,  \\
         v_t&=cu-ev+d_v\Delta v 
    \end{array}
    \right.
    \label{eq:lin_at_1}
\end{equation}
Let $-\lambda_k$, $k\in \N$ denote the eigenvalues of the Laplacian operator subject to Neumann Boundary conditions. The following proposition holds.
\begin{prop}
    Solutions of \Cref{eq:lin_at_0} satisfy the exponential decay
    \begin{equation}
    \label{ineq:expdecayL2}
        ||(u,v)(t)||_{L^2}\leq e^{-\mu t}||(u,v)(0)||_{L^2}
    \end{equation}
     for some constants  $C>0$ and $\mu>0$. Moreover, if
    \begin{equation}
    \label{eq:CondStaUnkat1}
     \forall k \in \N,\qquad  d_ud_v\lambda_k^2+\big(d_ue+(1-\gamma)d_v-bc\big)\lambda_k+(1-\gamma)e>0,   
    \end{equation} 
    then solutions of \Cref{eq:lin_at_1} satisfy also the exponential decay  \cref{ineq:expdecayL2}. Otherwise, there exist initial conditions such that
    \begin{equation}
    \label{ineq:expincreaseL2}
        ||(u,v)(t)||_{L^2}\longrightarrow +\infty \mbox{ as } t \longrightarrow +\infty,
    \end{equation}
    due to a finite number of positive eigenvalues.
\end{prop}
\begin{proof}
    We write
    \[u(x,t)=\sum_{k=0}^{+\infty}u_k(t)\varphi_k(x),\,\,v(x,t)=\sum_{k=0}^{+\infty}v_k(t)\varphi_k(x), \]
    where $\varphi_k, k \in \N,$ , denote the normalized eigenfunctions of the Laplacian with Neumann boundary conditions.
    The solutions of \Cref{eq:lin_at_0} satisfy (for $\lambda_k\neq \frac{e-a\gamma}{d_u-d_v}$)
    \begin{equation}
u_k(t)=e^{-(a\gamma+d_u\lambda_k)t}u_k(0), \,\,v_k(t)=e^{-(e+d_v\lambda_k)t}v_k(0)+c\frac{e^{(-a\gamma-d_u\lambda_k)t}  -e^{-(e+d_v\lambda_k)t}}{e-a\gamma+(d_v-d_u)\lambda_k}u_k(0).
\end{equation}
For \Cref{eq:lin_at_1}, we need to be more careful. The solutions of \Cref{eq:lin_at_1} satisfy
\begin{equation}
(E_k)\left\{
    \begin{array}{rl}
         u'_k&=(a(\gamma-1)-d_u\lambda_k)u_k+b\lambda_k v_k,  \\
         v'_k&=cu_k-(e+d_v\lambda_k) v_k 
    \end{array}
    \right.
    \label{eq:lin_at_1_k}
\end{equation}
which we can also write as
\begin{equation*}
    U_k'=A_kU_k
\end{equation*}
where $(U_k=(u_k,v_k)^{\mathsf T})$.
 The behavior of the PDE can be studied directly through the analysis of the systems $(E_k)$. Alternatively, one may use semigroup theory. Here, we provide a direct approach. We omit some computational details and focus on the main ideas.The eigenvalues associated to $E_k$ are
\[r_k^-=\frac{1}{2}\left(T_k-\sqrt{T_k^2-4D_k}\right),
\qquad
r_k^+=\frac{1}{2}\left(T_k+\sqrt{T_k^2-4D_k}\right),\]
with
\[T_k=a(\gamma -1) - (d_u+d_v)\lambda_k, \quad D_k=d_ud_v\lambda^2+(ed_u+a(1-\gamma)d_v)\lambda_k+a(1-\gamma)e.\]
Elementary computations show that $r_k^-$ and $r_k^+$ both converge to $-\infty$ as $k \to+\infty$. To prove the proposition it is therefore sufficient to prove that for $k$ large enough we have a uniform exponential decay:
\begin{equation}
\label{ineq:uniexpdecay}
    \exists C>0, \exists \epsilon>0,\exists K>0,\forall k>K,\quad ||U_k(t)||\leq Ce^{-\epsilon t}||(U_k(0)||.
\end{equation}
We provide the main ideas and omit the computational details. The proof splits into two cases. If $d_u\neq d_v$, we write
\begin{equation*}
    A_k=-\lambda_kM+R
\end{equation*}
with
\begin{equation*}
    M=\begin{pmatrix}
        d_u& -b\\
        0&d_v
    \end{pmatrix}
    \qquad 
    R=\begin{pmatrix}
       a( \gamma-1)& -b\\
        c&d-e
    \end{pmatrix}
\end{equation*}
and use the fact that 
\begin{equation*}
 \lim_{k\rightarrow+\infty}||\frac{A_k}{\lambda_k}+M||=0
\end{equation*}
to obtain \Cref{ineq:uniexpdecay}
If $d_u=d_v=d$, we write
\begin{equation*}
    A_k=-\lambda_kdI+B_k
\end{equation*}
and use the fact that
\begin{equation*}
    e^{A_kt}=e^{-\lambda_k d t}e^{B_kt}
\end{equation*}
since $-\lambda_kdI$ and $B_k$ commute.

\end{proof}
\textbf{Local Stability for the Nonlinear \Cref{eq:mKS}}
In this paragraph, we establish the local stability of the homogeneous stationary solutions of \Cref{eq:mKS}. We rewrite \Cref{eq:mKS} around each stationary solutions. Around $(0,0)$, the system reads
\begin{equation}
\left\{
    \begin{array}{rl}
         u_t&=-a\gamma u+d_u\Delta u+ (1+\gamma)u^2-u^3-b\nabla u\cdot\nabla v-bu\Delta v,  \\
         v_t&=cu-ev+d_v\Delta v. 
    \end{array}
    \right.
    \label{eq:NonLin_at_0}
\end{equation}
whereas around $(1,\frac{c}{e})$, it reads
\begin{equation}
\left\{
    \begin{array}{rl}
         u_t&=a(\gamma-1)u-b\Delta v +d_u\Delta u+ a(\gamma-2)u^2-au^3-b\nabla u\cdot\nabla v-bu\Delta v,,  \\
         v_t&=cu-ev+d_v\Delta v 
    \end{array}
    \right.
    \label{eq:NonLin_at_1}
\end{equation}
We rewrite both \eqref{eq:NonLin_at_0} and \eqref{eq:NonLin_at_1} as 
\begin{equation}
    U_t=\mathcal{A}_*U+\mathcal{F}_*(U)
\end{equation}
where \begin{equation}
    \mathcal{F_*}(U)
    =
    \begin{pmatrix}
        a(1+\gamma-3u_*)u^2-au^3
        -b\nabla u\cdot\nabla v
        -bu\Delta v\\[1mm]
        0
    \end{pmatrix}.
    \label{eq:nonlinear_remainder}
\end{equation}

and
\begin{equation}
    \mathcal{A}_*
    \begin{pmatrix}
        u\\
        v
    \end{pmatrix}
    =
    \begin{pmatrix}
        d_u\Delta u+f'(u_*)u-bu_*\Delta v\\[1mm]
        d_v\Delta v+cu-ev
    \end{pmatrix}.
    \label{eq:linearized_operator}
\end{equation}
with $u_*=0$ for \Cref{eq:NonLin_at_0}, $u_*=1$ for \Cref{eq:NonLin_at_1}.
The following local-stability result holds.
\begin{thm}
There exist constants $\eta>0$ and $\mu>0$ such that if,
\[\|U(0)\|_{H^2}<\eta\]
then the solution of \Cref{eq:NonLin_at_0} satisfies
    \begin{equation}
    \label{ineq:NonLinexpdecayL2H2}
        ||U(t)||_{H^2}\leq e^{-\mu t}||U(0)||_{H^2}
    \end{equation}
    Furthermore, if condition \eqref{eq:CondStaUnkat1} holds then solution of \Cref{eq:NonLin_at_1}  also satisfies the exponential decay \eqref{ineq:NonLinexpdecayL2H2}.
\end{thm}
\begin{proof}
   The proof relies on the fact that, under the assumptions of the theorem, the operator $\mathcal{A}$ is sectorial on

\[X=L^2(\Omega)\times L^2(\Omega) \]
with domain
\[ X^1=\mathcal{D}(A)\subset H^2(\Omega)\times H^2(\Omega).  \]

Since $n\leq 3$, the following Sobolev embedding 
\[
H^2\hookrightarrow L^\infty(\Omega).
\]
holds.
Exponential stability of the linear semigroup generated by $\mathcal{A}$ gives
\[\|e^{-At}\|_{L(X)}\leq Me^{-\mu t}\]
and, by analyticity of the semigroup,
\[\|Ae^{-At}\|_{L(X)}\leq M_1t^{-1}e^{-\mu t}.\]
Moreover, the Sobolev embedding and the structure of $\mathcal{F}$ imply, for $U$ sufficiently small in $X^1$,

\[
\|F(U)\|_{X}\leq C_1\|U\|^2_{X_1}+C_2\|U\|^3_{X_1}\leq C\|U\|^2_{X} \mbox{ for } U \mbox{ small enough in } X_1
\]
The result then follows from the variation-of-constants formula
\[
U(t)=e^{t\mathcal{A}}U(0)
+\int_0^t e^{(t-s)\mathcal{A}}\mathcal{F}(U(s)),ds,
\]
together with a standard bootstrap and continuity argument. We refer to \cite{amann1990dynamic,henry1981geometric,pazy1983semigroups} for standard computational details.
\end{proof}
\section{Numerical Simulations and Emergence of angiogenic patterns}
\label{sec:NSAP} In this section, we describe how numerical simulations of \eqref{eq:mKS} can lead to angiogenic patterns and how variation of the parameter $\gamma$ affects those patterns. For numerical simulations, we consider a space-discretized version of \Cref{eq:mKS} which corresponds to a finite difference scheme in space (with space step $h=1$) and a Runge-Kutta resolution method in time.
\subsection{Parameter setting and Emergence of Angiogenic patterns}
 The value of parameters are as follows:
\begin{equation}
         a=7, \, b=10,\, d_u=1,\,   c=3, \, e=2,\, d_v=10.
    \label{eq:param-mKS}
\end{equation}
The parameter $\gamma$ is to be varied between $0$ and $1$. 
The model was initially inspired by \cite{Mim1996,Kut2012} and subsequently adjusted to capture angiogenic patterns. We choose the initial conditions as follows. The initial concentration of $u$ is distributed over the square domain: it is set to $0.8$ in nine small squares that are homogeneously distributed across the spatial domain, and to $0.2$ elsewhere. Additionally, a uniformly distributed random perturbation in the range $[0, 0.2]$ is added at each point. For $v$, the initial concentration is set to the constant value $0.5$.

The underlying idea is that the concentrations are driven toward $0$ or $1$ by the reaction function $f$, in combination with diffusion and transport terms. Choosing $\gamma$ in the range $0.25$–$0.29$ leads to the formation of angiogenic patterns.

\Cref{fig:MicroAndKS} illustrates a microscope image of BMEC, along with a corresponding numerical result obtained from \Cref{eq:mKS}. The left panel shows the image of the well containing BMEC after 8 hours. The center panel displays the same image after filtering using the AngioTool software \cite{AngiotoolWS, Zud2011}. The right panel presents the solution of \Cref{eq:mKS} with $\gamma = 0.29$ at a fixed time. More details about the simulations are provided in the next paragraph.

\begin{figure}
    \centering
    \includegraphics[height=3cm,width=4cm]{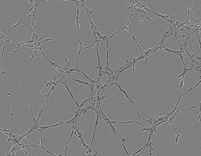}
    \includegraphics[height=3cm,width=4cm]{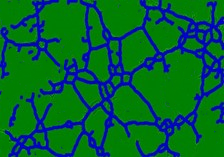}
    \includegraphics[height=3cm,width=4cm]{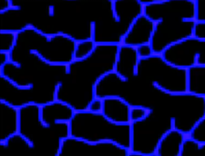}
    \caption{Left: image of BMEC provided by microscopy after 8 hours of evolution.  Center: transformed image provided by the free software angiotool. Right: solution of the discretized modified KS model \eqref{eq:mKS} at a fixed time.}
    \label{fig:MicroAndKS}
\end{figure}
\subsection{Variation of the parameter $\gamma$}
We proceed with the variation of the parameter $\gamma$ and discuss the variety of observed patterns. If one had to consider only the underlying ODE \cref{eq:mKSODE}, the principle would be as follows: the smaller the value of $\gamma$, the more blue, (i.e., closer to the constant steady state $(u,v) = (1, \frac{c}{e})$) the final pattern would be. This is because decreasing $\gamma$ enlarges the basin of attraction of the steady state $u = 1$ for initial conditions in $(0,1)$. However, this intuition does not always hold in the full model. We begin with the value $\gamma = 0.1$, illustrated in \Cref{fig:KSgam01}. For this small value of $\gamma$, the solution asymptotically converges to the spatially homogeneous steady state with $u = 1$ and $v = \frac{c}{e}$. Next, we consider $\gamma = 0.2$, corresponding to \Cref{fig:KSgam02}. The dynamics are very similar to the previous case, at least initially. Interestingly, a small black region (i.e. close to the steady state $(u,v) = (0,0)$) persists in the lower right corner of the domain, leading to a traveling wave solution. Asymptotically, the domain becomes almost entirely black, except for thin regions near the boundary. Simulations for $\gamma=0.25$ are shown in \Cref{fig:KSgam025}. In this case, the solution initially evolves rapidly toward a state dominated by blue, with small regions of black. Subsequently, the black regions propagate, leading to structures similar to those observed in experiments. The patterns then evolve more slowly, with small polygons shrinking and eventually disappearing after transitioning through a blue phase. The final state is a network that suggests a stable, minimal-energy configuration, corresponding to a spatially non-homogeneous steady-state solution. This description indicates a pronounced temporal structure with distinct dynamical phases. We will discuss this point in more detail in section \cref{sec:NS3timephase} by examining the time evolution at specific spatial positions. Finally, for $\gamma = 0.29$ (see \Cref{fig:KSgam029}), the solution initially evolves rapidly toward a state with a mixture of blue and black regions. From there, the area occupied by black regions tends to increase, leading to a structure similar to those observed in experiments. Subsequently, as before, the patterns evolve slowly, with small polygons gradually shrinking and disappearing. The resulting stationary pattern resembles a degenerate network of cells.
\begin{figure}
    \centering
    \includegraphics[height=3cm,width=4cm]{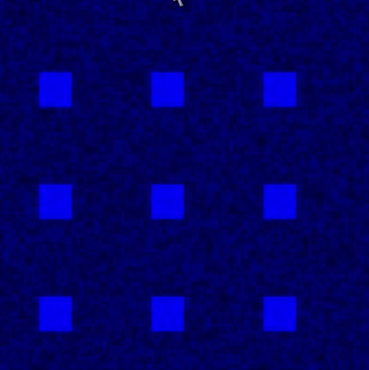}
    \includegraphics[height=3cm,width=4cm]{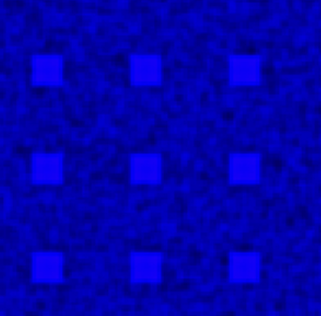}
    \includegraphics[height=3cm,width=4cm]{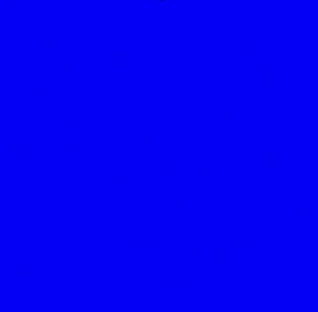}
    \caption{Solution of the modified KS model \eqref{eq:mKS} with $\gamma=0.1$ and $t\in \{0,0.6,2.4\}$, from left to right. Only $u$ is represented.  The solution evolves toward the constant stationary solution $(u,v)=(1,\frac{c}{e})$.}
    \label{fig:KSgam01}
\end{figure}

\begin{figure}
    \centering
    \includegraphics[height=3cm,width=4cm]{CIGam01.png}
    \includegraphics[height=3cm,width=4cm]{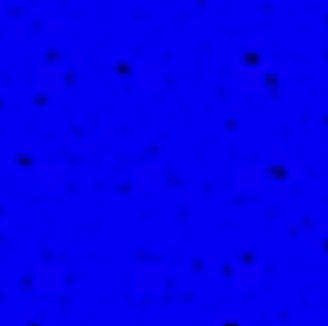}
    \includegraphics[height=3cm,width=4cm]{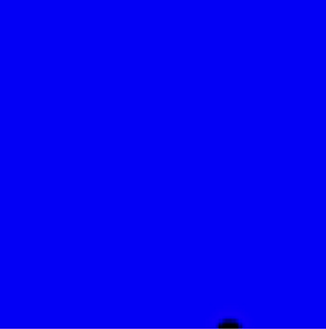}\\
    \includegraphics[height=3cm,width=4cm]{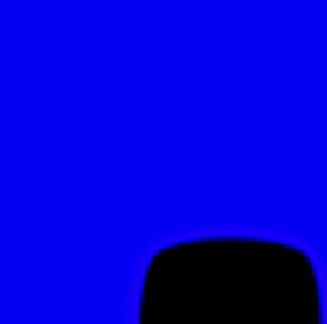}
    \includegraphics[height=3cm,width=4cm]{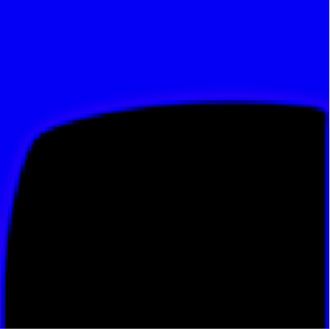}
    \includegraphics[height=3cm,width=4cm]{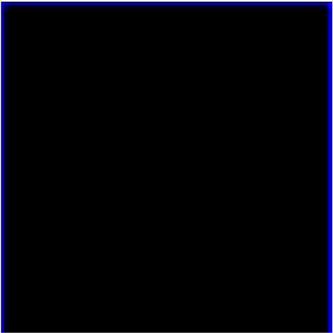}
    \caption{Solution of the modified KS model \eqref{eq:mKS} with $\gamma = 0.2$ at times $t \in {0, 2, 4.6, 28.2, 67, 180}$, shown from left to right and top to bottom. Only $u$ is represented. After a very short transient period, the solution evolves toward the constant stationary state $(u, v) = (1, \frac{c}{e})$ almost everywhere, except for a small region near the bottom-right corner, where it remains at $(0,0)$. From there, a traveling-wave–type behavior is observed: the solution progressively decreases toward $(0,0)$, propagating from the bottom-right corner toward the top-left corner. Along the left, top, and right boundaries, the solution remains close to $0.9$. }
    \label{fig:KSgam02}
\end{figure}

\begin{figure}
    \centering
    \includegraphics[height=3cm,width=4cm]{CIGam01.png}
    \includegraphics[height=3cm,width=4cm]{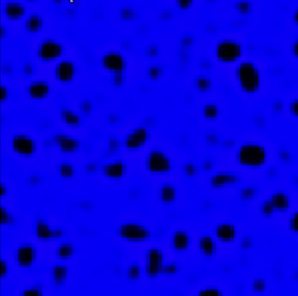}
    \includegraphics[height=3cm,width=4cm]{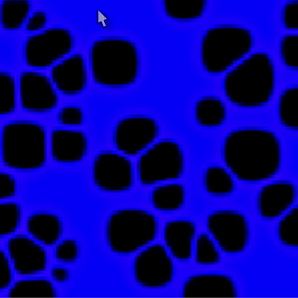}\\
    \includegraphics[height=3cm,width=4cm]{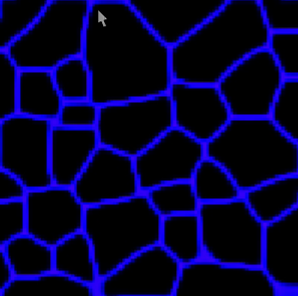}
    \includegraphics[height=3cm,width=4cm]{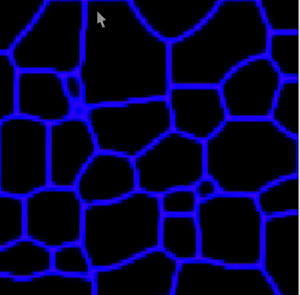}
    \includegraphics[height=3cm,width=4cm]{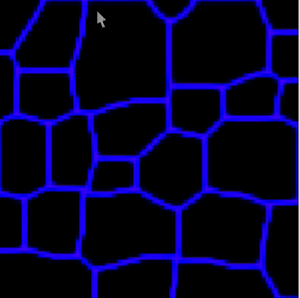}
    \caption{Solution of the modified KS model \eqref{eq:mKS} with $\gamma = 0.25$ at times $t \in {0, 2.5, 7.8, 19.2, 46.3, 180}$, shown from left to right and top to bottom. Only $u$ is represented. The solution initially evolves rapidly toward a state dominated by blue regions (close to the constant stationary solution $(u, v) = (1, \frac{c}{e})$), with small patches of black regions (close to $(u, v) = (0,0)$). From there, the black regions propagate, forming a network-like structure of cells similar to those observed in experiments. The patterns then evolve slowly, with small polygons gradually shrinking and disappearing after a transition through blue. The final state is a network that appears to correspond to a stable minimal-energy configuration, representing the non-homogeneous stationary solution.}
    \label{fig:KSgam025}
\end{figure}

\begin{figure}
    \centering
    \includegraphics[height=3cm,width=4cm]{CIGam01.png}
    \includegraphics[height=3cm,width=4cm]{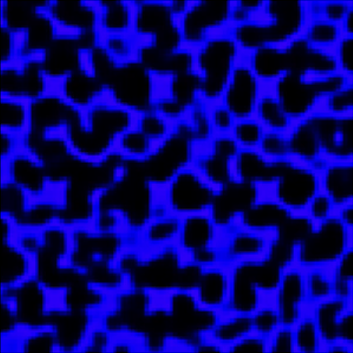}
    \includegraphics[height=3cm,width=4cm]{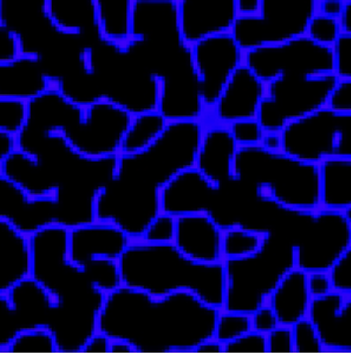}\\
    \includegraphics[height=3cm,width=4cm]{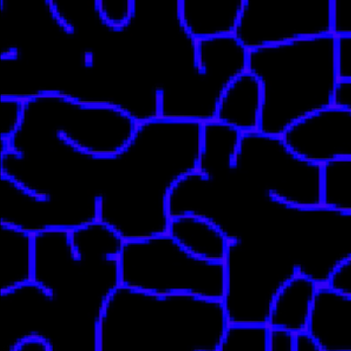}
    \includegraphics[height=3cm,width=4cm]{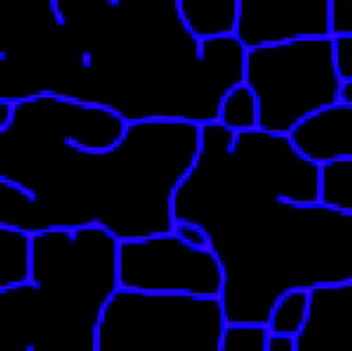}
    \includegraphics[height=3cm,width=4cm]{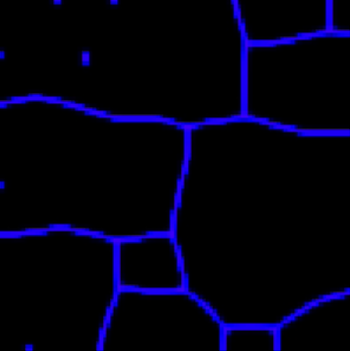}
    \caption{Solution of the modified KS model \eqref{eq:mKS} with $\gamma = 0.29$ at times $t \in {0, 2.8, 6.5, 18.5, 32.4, 180}$, shown from left to right and top to bottom. Only $u$ is represented. The solution initially evolves rapidly toward a state with both blue and black regions. From there, the black regions gradually expand, forming a cell-like network similar to those observed in experiments. The patterns then evolve slowly, with small polygons shrinking and disappearing, resulting in a degenerated network of cells.}
    \label{fig:KSgam029}
\end{figure}

\begin{figure}
    \centering
    \includegraphics[height=3cm,width=4cm]{CIGam01.png}
    \includegraphics[height=3cm,width=4cm]{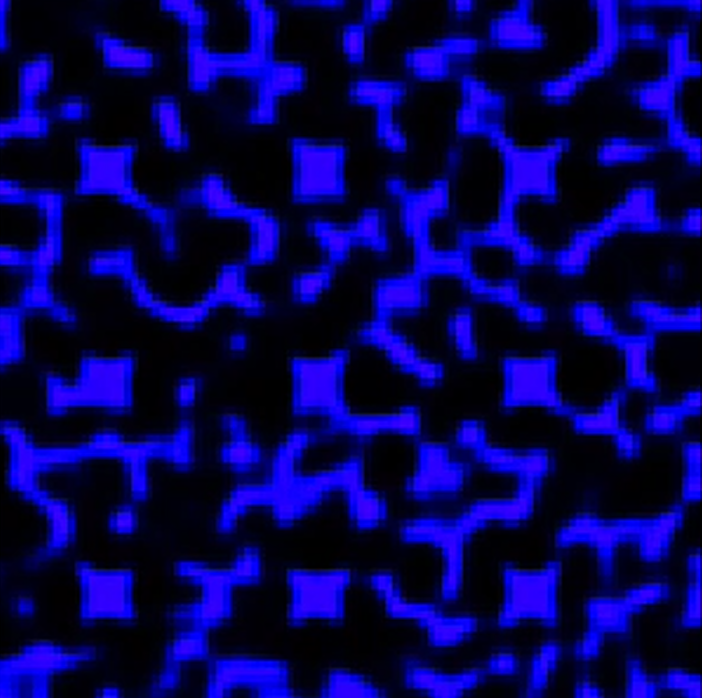}
    \includegraphics[height=3cm,width=4cm]{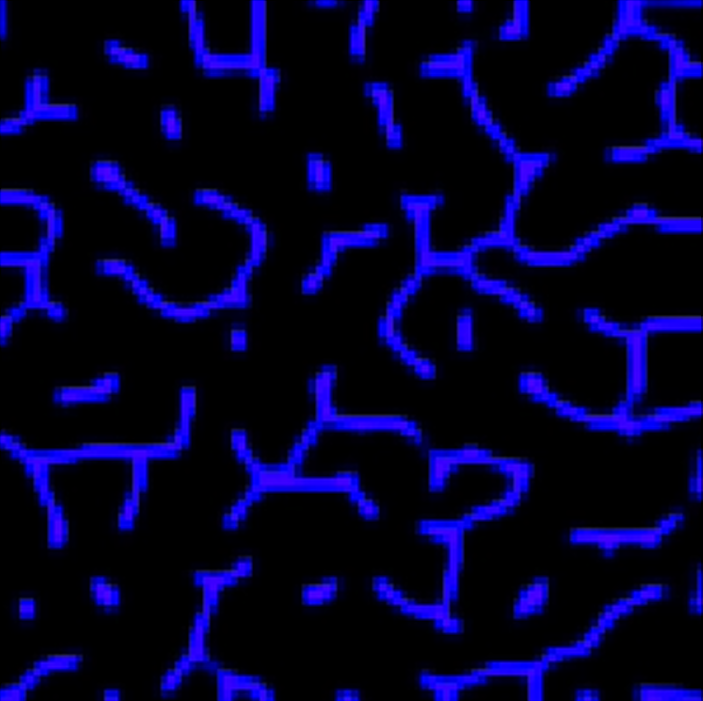}\\
    \includegraphics[height=3cm,width=4cm]{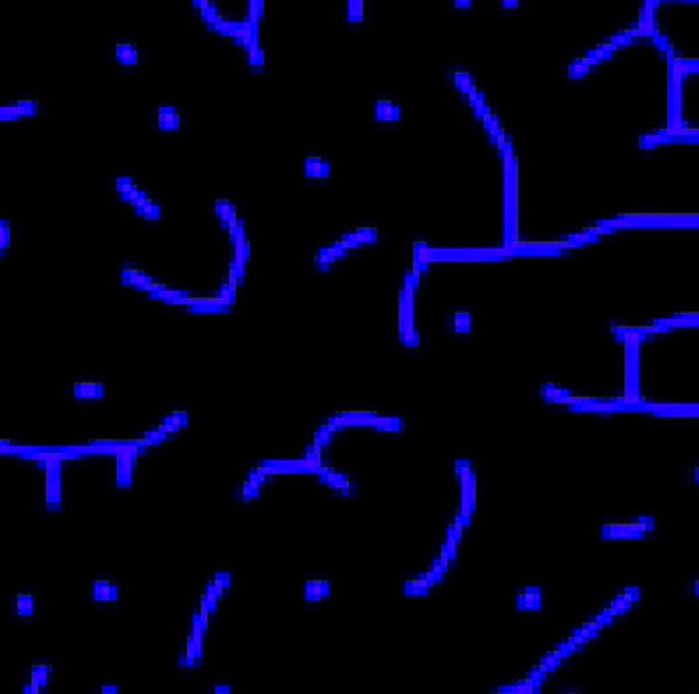}
    \includegraphics[height=3cm,width=4cm]{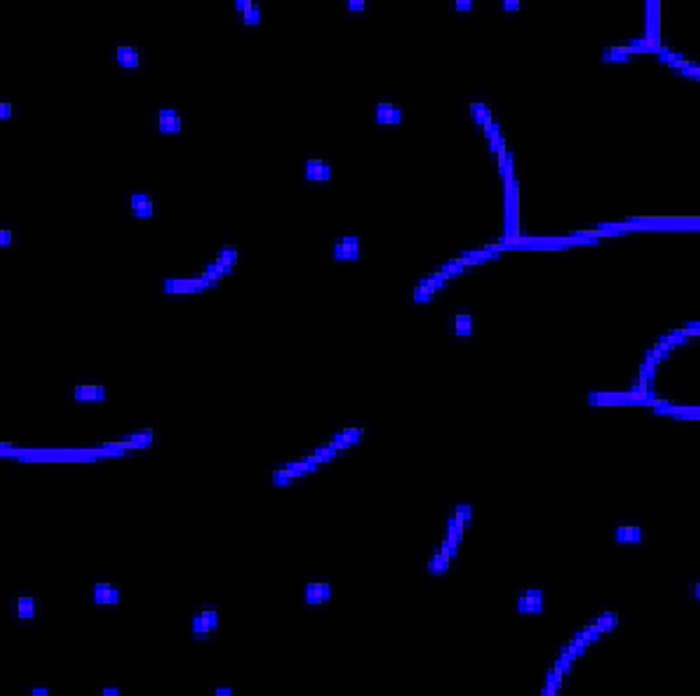}
    \includegraphics[height=3cm,width=4cm]{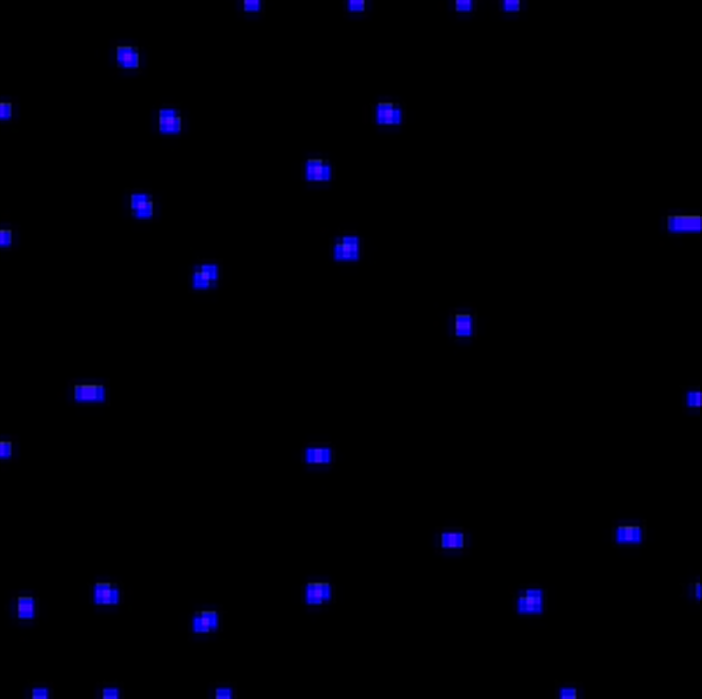}
    \caption{Solution of the modified KS model \eqref{eq:mKS} at six fixed times spanning $t = 0$ to $180$, shown from left to right and top to bottom. Only $u$ is represented. The solution initially evolves rapidly toward a state with both blue and black regions. In this case, no polygons form. At the final stage, a few blue spots remain on a predominantly black background.}
    \label{fig:KSgam03}
\end{figure}
\section{Data Driven reproduction of Chemo-attractant evolution}
In this section, we describe a procedure to reconstruct the temporal evolution of the cell patterns by combining elements of Eq. \eqref{eq:mKS} with experimentally observed data. To this end, fifteen microscopy images were used to generate the reconstruction; six representative images are shown in Fig.~\ref{fig:original_images}. The resulting movie illustrating the evolution of the cellular patterns is provided as supplementary material accompanying the manuscript.

\begin{figure}
    \centering
        \includegraphics[height=3cm,width=4cm]{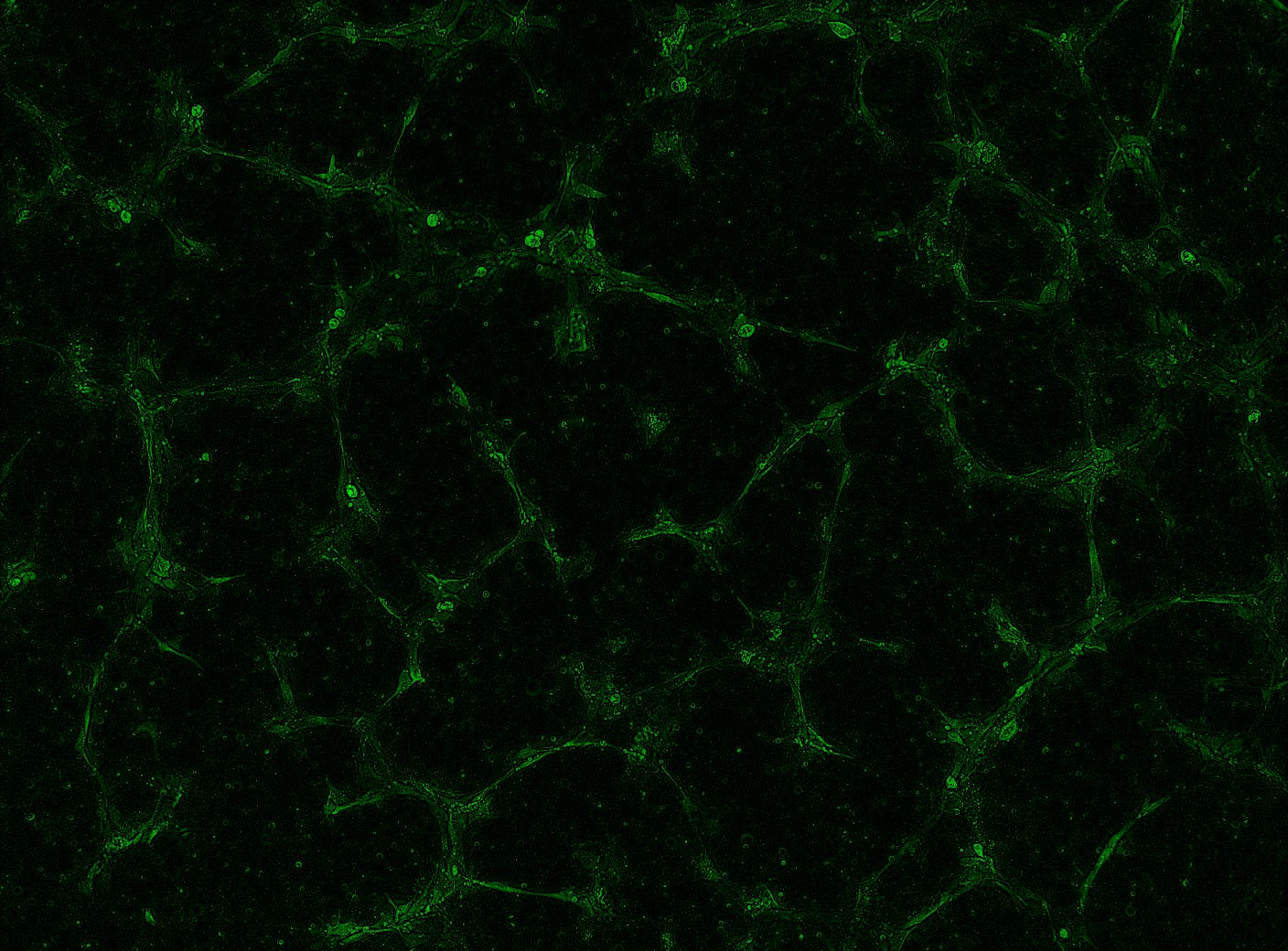}
            \includegraphics[height=3cm,width=4cm]{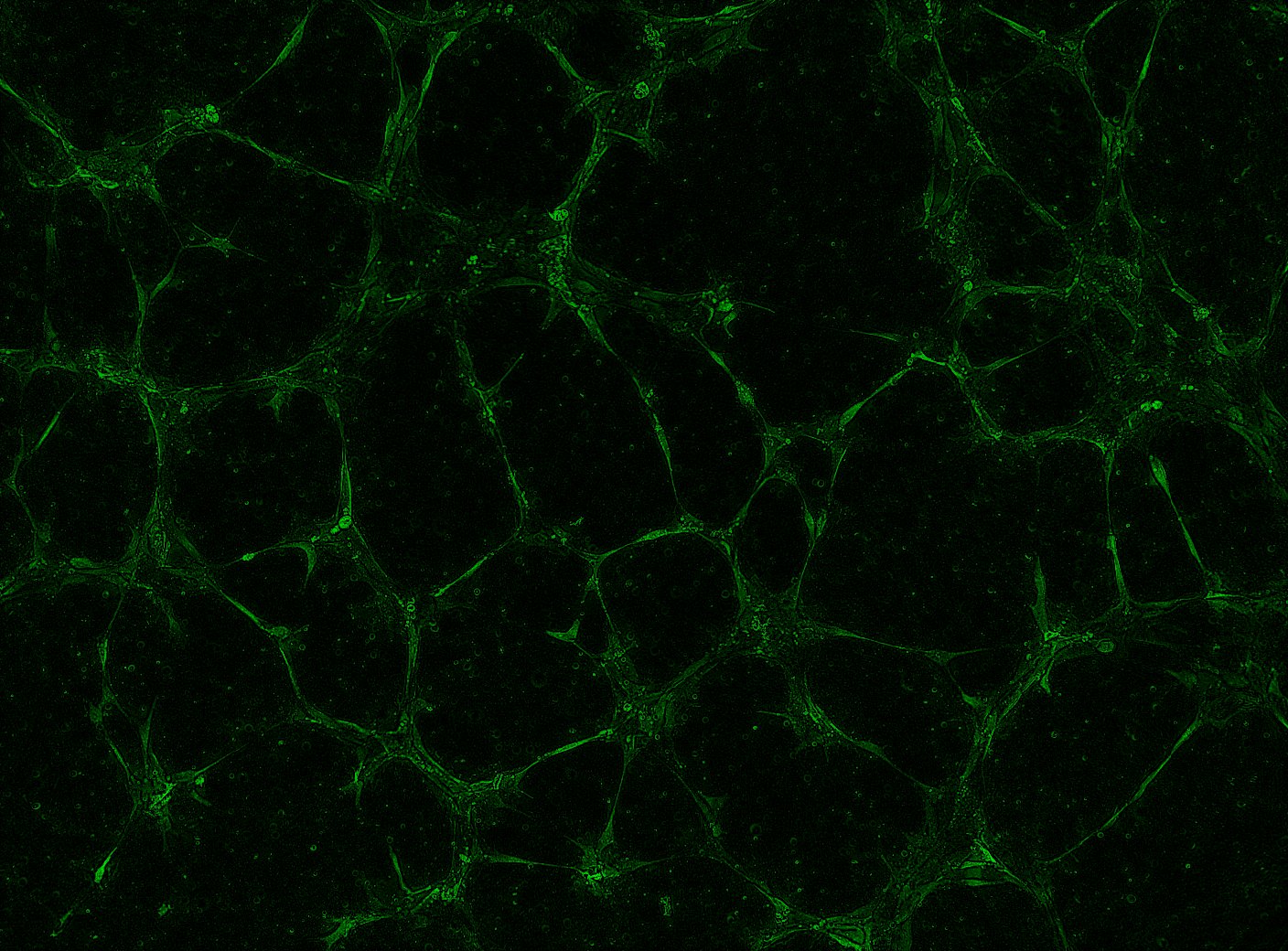}
                \includegraphics[height=3cm,width=4cm]{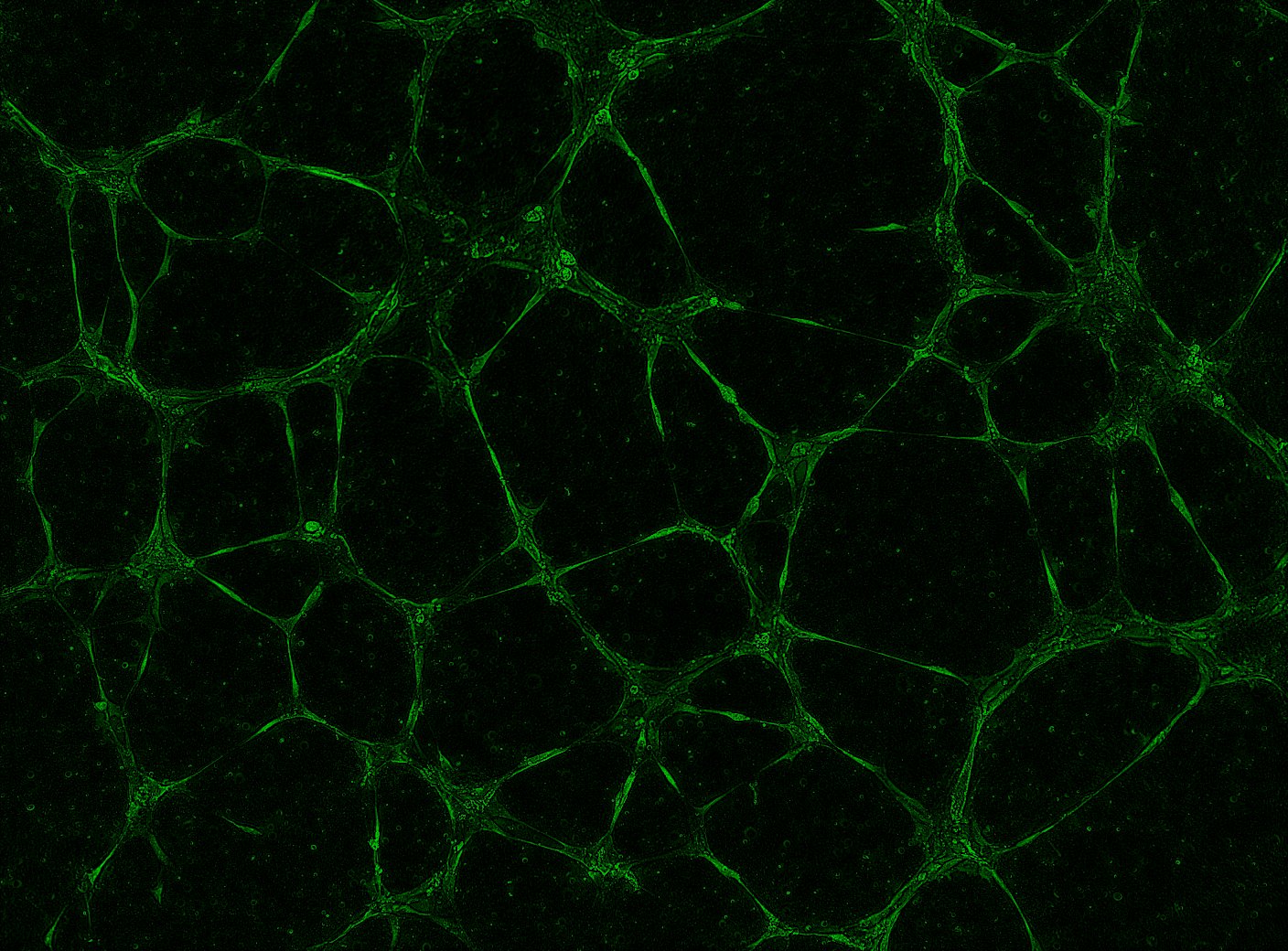}\\
                 \includegraphics[height=3cm,width=4cm]{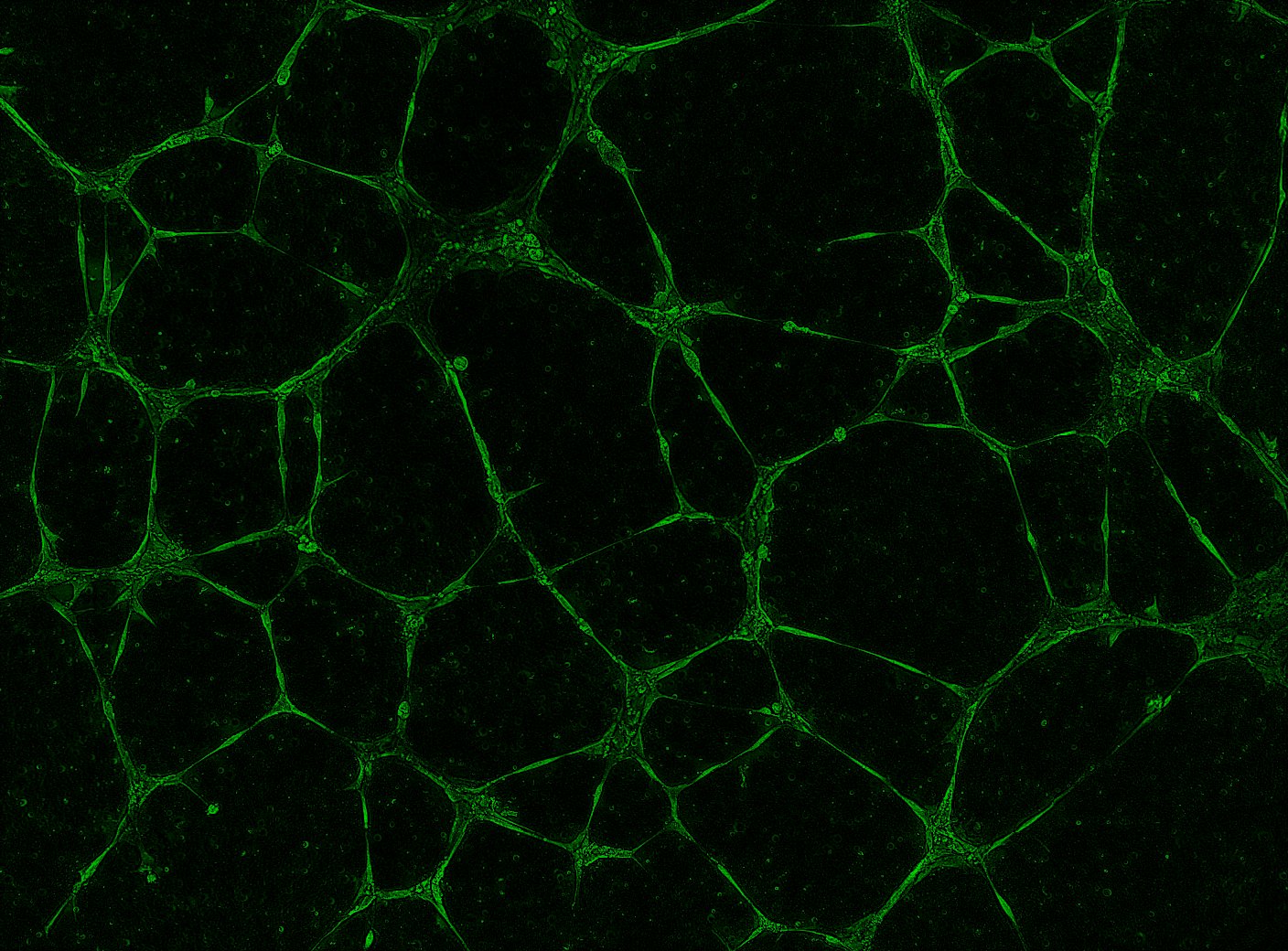}
        \includegraphics[height=3cm,width=4cm]{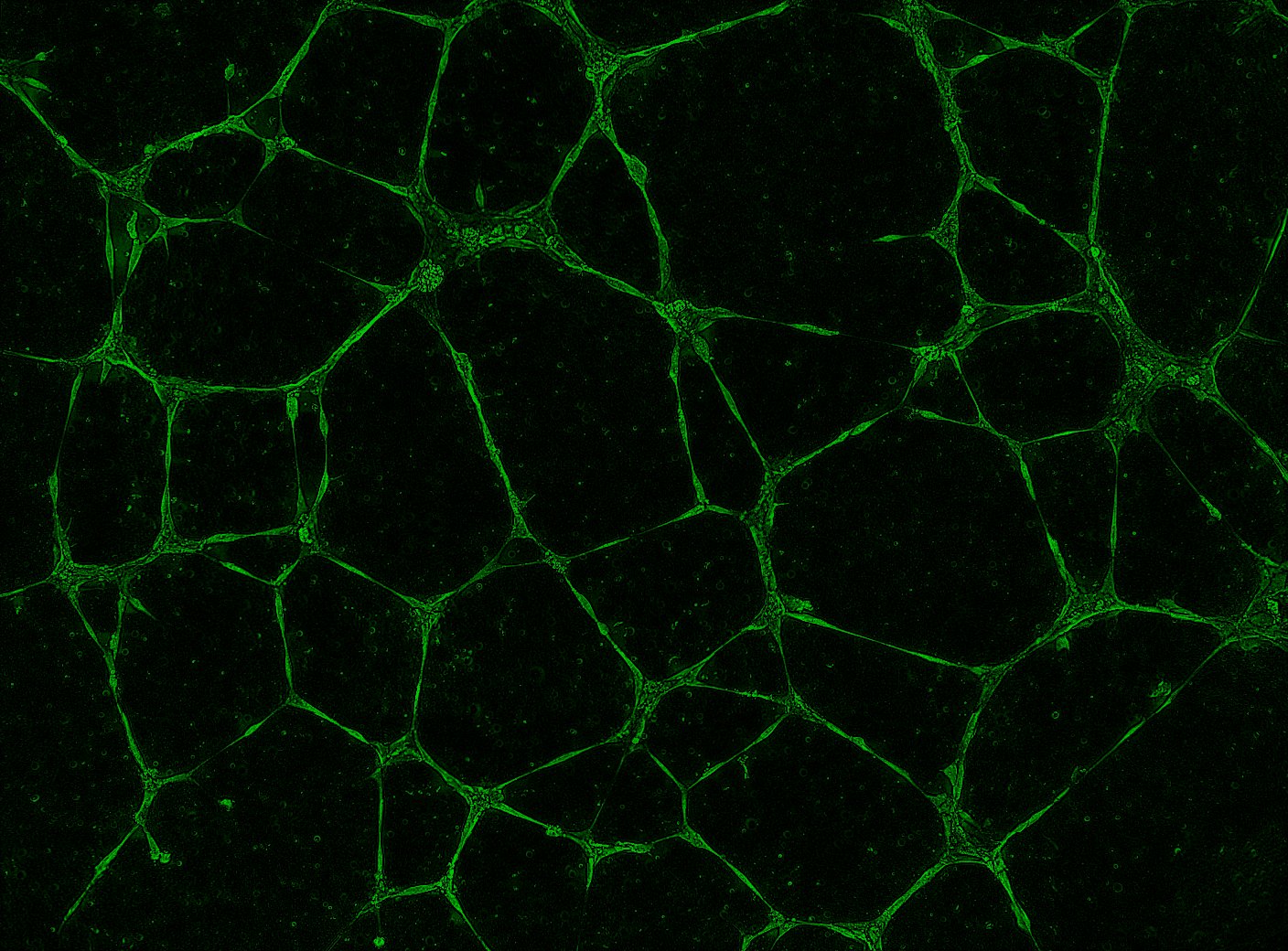}
            \includegraphics[height=3cm,width=4cm]{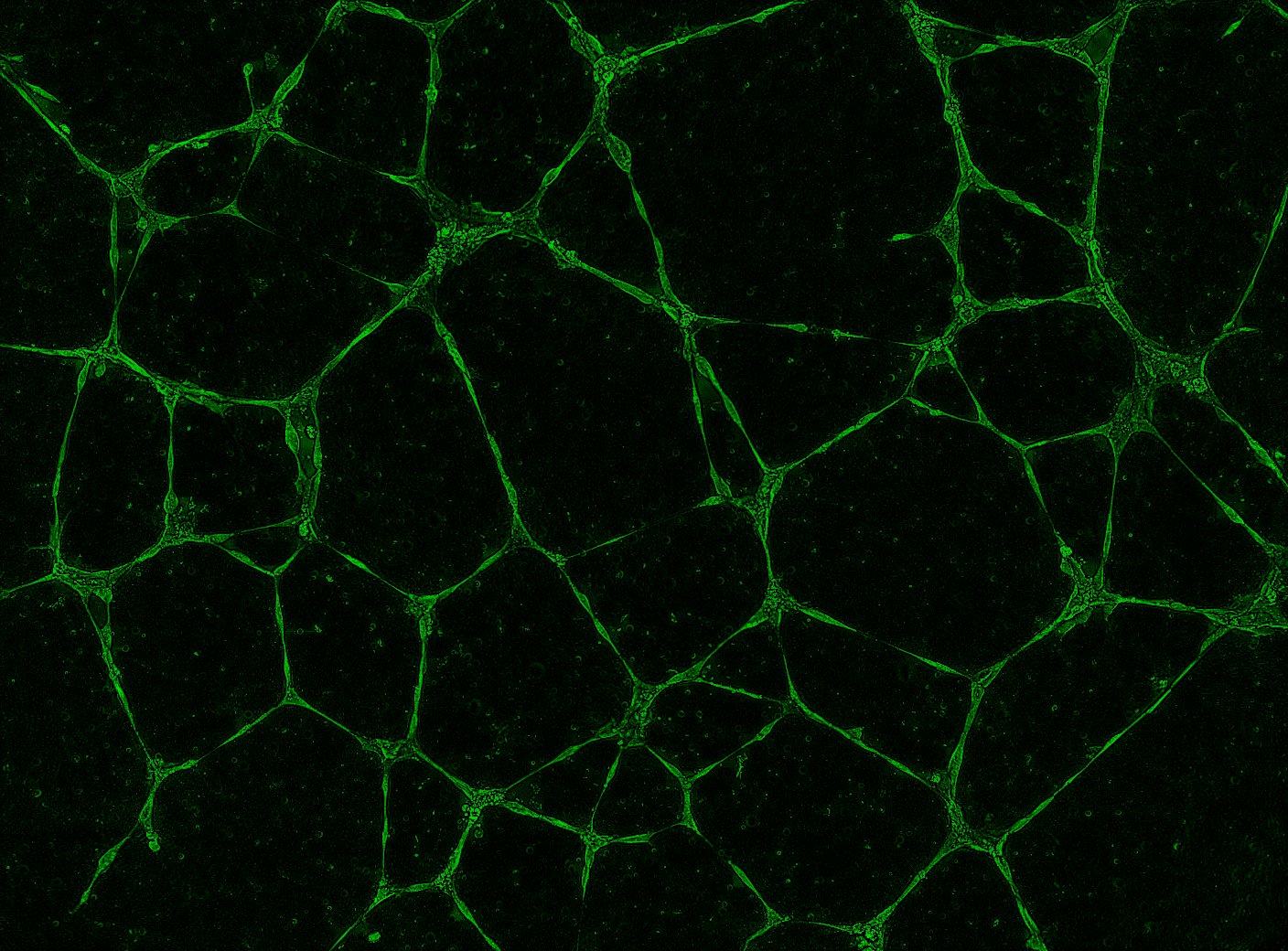}

    \caption{Original images of BMEC obtained by microscopy at the Elahi Lab at times t=2,4,6,8,10 and 12 hours.}
    \label{fig:original_images}
\end{figure}
Our objective is to find the value of fields $v$ that correspond to the temporal evolution of the cellular patterns. To this end, we treat two consecutive microscopy images, denoted by $u(t)$ and $u(t+1)$, as given data and compute the corresponding field 
$v$ by solving the equation
\begin{equation}
         u(t+1)-u(t)=f(u(t))-b\nabla \cdot (u(t)\nabla v(t))+d_u\Delta u(t)  
\end{equation}
or equivalently
\begin{equation}
\label{eq:elliptic_v}
-\nabla u(t)\cdot \nabla v(t)-u\Delta v(t)=\frac{1}{b}\Big( u(t+1)-u(t)-f(u(t))-d_u\Delta u(t) \Big)         
\end{equation}
Equation \eqref{eq:elliptic_v} is an elliptic partial differential equation for 
$v$. Since $u>0$, classical elliptic theory ensures the existence and uniqueness of the solution up to an additive constant. However, solving this equation numerically remains challenging. To address this issue, we adapted the GMRES (Generalized Minimal Residual) method, an iterative Krylov subspace algorithm that computes approximate solutions by minimizing the residual norm at each iteration \cite{SaadSchultz1986}. The details of the numerical procedure are provided in the Appendix, and the corresponding C++ implementation is available. The six obtained fields $v$ at $t=2,4,6,8,10$ and $12$ are illustrated in figure \ref{fig:fields_v} .

\begin{figure}
    \centering
        \includegraphics[height=3cm,width=4cm]{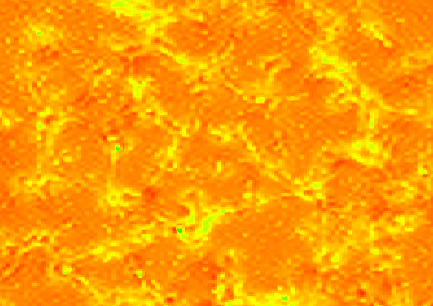}
            \includegraphics[height=3cm,width=4cm]{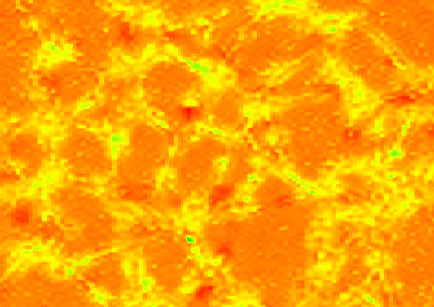}
                \includegraphics[height=3cm,width=4cm]{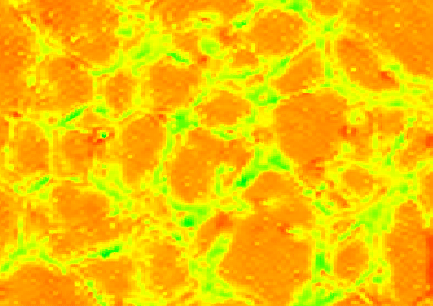}\\
                 \includegraphics[height=3cm,width=4cm]{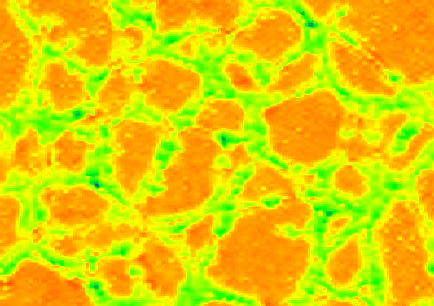}
        \includegraphics[height=3cm,width=4cm]{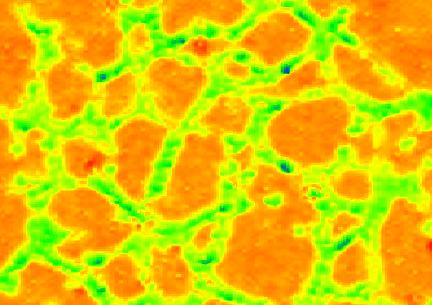}
            \includegraphics[height=3cm,width=4cm]{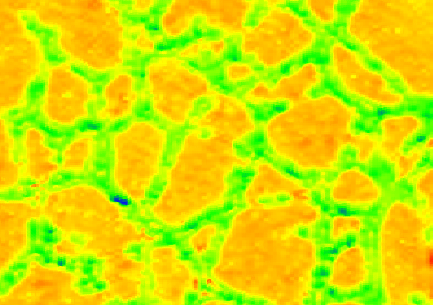}

    \caption{The six fields $v$ obtained by solving equation \eqref{eq:elliptic_v} at $t=2,4,6,8,10$ and $12$ hours.}
    \label{fig:fields_v}
\end{figure}

\section{Numerical Simulations: a three time-phase structure for $\gamma=0.25$}
\label{sec:NS3timephase}
In this section, we further proceed with a more in-depth qualitative analysis of the numerical solutions for $\gamma=0.25$. Our goal is to investigate what are the dynamical mechanisms that lead to the patterns observed in the previous section for this particular value. Since in the setting considered here, the stationary patterns appear visually as networks of a blue color in a black background, we focus on how solutions evolve in time, at some fixed spatial close positions, where $u$ is asymptotically close to its maximum value (close to the value one which appears in blue)  or close to its minimum value (close to zero which appears in black). Interestingly, when looking at the time evolution at these fixed positions, we can distinguish three phases in time. This seems to correspond to dynamics at different time scales: very fast, fast and slow. Biological processes with different time scales have been widely observed and reported, specifically in neuroscience context. It has also been studied theoretically. We refer for example to the textbook \cite{BookKuehn} and the numerous references cited in there; the topic is commonly referred as slow-fast dynamics in the field of dynamical systems and computational neuroscience. A detailed investigation of the theoretical slow-fast structure of the system under consideration is not in the scope of this article. \Cref{fig:uvtgam025} illustrates the time evolution of relevant quantities: in the left panels, we represent the time evolution of $u$, $\Delta u$ and $\nabla u \cdot \nabla v$, while in the right panels, we illustrate the time evolution of $v$ and $\Delta v$. This allows us to get a glimpse about the factors that drive the dynamics locally in space. To highlight the mechanisms at play, we describe hereafter in some detail two opposite dynamical behaviors (at fixed space positions) for which $u$ is asymptotically close to its maximum and minimum values (see \Cref{fig:uvtgam025}, first and third row). Analog arguments hold for dynamics at other positions. The first row of \Cref{fig:uvtgam025} illustrates the dynamics of the above mentioned quantities at the space point $(50,50)$ for which $u$ is asymptotically at a local maximum. We can distinguish three phases. During the first phase, $u$ increases rapidly. Analog behavior is observed for $v$. The term $\nabla u \cdot \nabla v$ is close to zero ( $i.e.$ $\nabla u$ and $\nabla v$ are almost orthogonal) and negligible. The term $\Delta  u$ is positive. This contributes to the fast increase in $u$. We note however that $\Delta  u$ decreases with time. This means that at this point, $u$ is smaller than the average values of its neighbors, but tends to increase faster than its neighbors. Meanwhile, analogously, $\Delta v$ is positive and decreases. At some point in time, we observe a sudden change in the time derivative of $u$. We associate this sudden change with the start of the second phase. After that, the variable $u$ is still increasing, but in a slower fashion, meanwhile, $\Delta  u$ becomes negative which contributes to slowing down $u$. The sudden change in the evolution of the variable $v$ is even more spectacular. The second phase starts indeed with a sudden change in the time derivative of $v$ from positive to negative: the variable $v$ starts to decrease abruptly during the second phase.  $\Delta v$ becomes negative and decreases. Negative values of $\Delta v$ contribute to the increase of $u$ and the decrease of $v$. Note that $b$ and $d_v$ have large values which intensifies the impact of $\Delta v$ values. At some point, we observe again an abrupt change in the time derivative for quantities $u$, $v$, $\Delta u$ and $\Delta v$: the time derivative becomes close to zero and we associate this with the beginning of the third phase. After that, all the quantities remain almost steady with a notable small increase in $u$ and $\Delta u$.  We note that $u$ and $v$ reach stationary positive values, which are local maximums. Accordingly, $\Delta u$ and $\Delta v$ reach stationary negative values. We note that at the equilibrium, for the first equation $u>1$ which means that $f(u)<0$, $\Delta u<0$ (negative contributions), $-bu\Delta v>0$ (positive contribution). Here is an effect of the chemotaxis term: a local maximal concentration in $v$ contributes to a local maximal concentration in $u$, and compensates the negative contributions of the growing term (here shrinking) $f(u)$ and the diffusive term $\Delta u$. The third row of \Cref{fig:uvtgam025} illustrates the dynamics at the spatial point $(x,y)=(52,50)$.  During the first phase, as described for $(x,y)=(50,50)$, $u$ and $v$ increase rapidly. The term $\nabla u \cdot \nabla v$ increases slowly. The term $\Delta  u$ has a small fast increase followed by a small fast decrease. This interestingly tells us about the growing dynamics of $u$ in comparison to its neighbors. Since $\Delta u>0$ during this phase, the quantity $u$ is lower than the average value of its neighbors; when $\Delta u$ increases it tells us that the neighbors are growing faster and when $\Delta u$ decreases it tells us that the quantity $u$ is growing faster at this space point than at the neighboring positions in average. At the end of phase $1$, it has a small positive value. We note that since $u\in (\gamma, 1)$, the growth factor $f(u)$ has an increasing effect. During phase 1, $\Delta v$ has a small increase followed by a small decrease and reaches a value close to zero.  In this case, we observe that the increase in $u$ is caused by the growth and diffusive factors. The chemotaxis terms have a small negative effect.  During phase 2, $u$ and $v$ have a fast decrease. This contrasts with the observed behavior at $(x,y)=(50,50)$. $\nabla u \cdot \nabla v$ continues to increase at a slow constant rate.   $\Delta  u$ increases rapidly and reaches a high positive value. $\Delta v$ has a small slow increase. From this observation, we deduce that $\nabla u \cdot \nabla v$ and $\Delta v$, $i.e$ the chemotaxis terms, are the factors that first contribute to the decrease of $u$. Even though, they seem to be small it is sufficient  because of the large values of $b$ and $d_v$. At some point, $u$ crosses $\gamma$ downwards which means that $f(u)$ starts also to contribute to the decrease in $u$.   During the last phase, $v$, $\Delta v$ and $\nabla u \cdot \nabla v$ remain constant. There are some slight evolutions for $u$ and $\Delta u$, but at some point all quantities reach a stationary state. By contrast with the observation made at $(x,y)=(50,50)$, we note that at the equilibrium, for the first equation, $u$ reach a quite low value with $u\in(0,\gamma)$, which means that $f(u)<0$. We note also that $-bu\Delta v<0$, $-b\nabla u \cdot \nabla v<0$. All there are negative contributions. The term $\Delta u>0$ contributes positively. We briefly mention the dynamics at the spatial point $(54,50)$ at which $u$ is at its local minimum, see the last row of \Cref{fig:uvtgam025}. Here, the dynamics are close to that described at $(x,y)=(52,50)$. The main difference being that $u$ is almost equal to $0$ and $\Delta u$ is  much smaller. We have also that $\nabla u \cdot \nabla v$ is almost $0$. At the end, we observe that $u(50,50)>u(51,50)>u(52,50)>u(53,50)>u(54,50)$.

% \Cref{fig:uvtgam025-2} illustrates the dynamics at the space point $(51,50)$. During the first phase, $u$ increases rapidly, as well as $v$. The term $\nabla u \cdot \nabla v$ increases slowly. The term $\Delta  u$ increases first rapidly from zero and then decreases rapidly but remains positive. This contributes to the rapid increase in $u$ among with the positive value of $f(u)$ for $u\in (\gamma,1)$. We recall that a positive value of the Laplacian means that at this point $u$ is smaller than the average values of its neighbors; its dynamics further indicate that it has first a slower increase than its neighbors and then a faster increase. During phase 1, $\Delta v$ increases by a little and then decreases. Phase 2 starts with a sudden change in the time derivative. $u$ starts to decrease slowly. $v$ starts to decrease. $\nabla u \cdot \nabla v$ continues to increase at a constant rate.   $\Delta  u$ decreases and reach negative values. So do $\Delta v$, it decreases and becomes negative. During the last phase, $v$, $\Delta v$ and $\nabla u \cdot \nabla v$ do not move so much. There is some evolution for $u$ and $\Delta u$, but at some point, all  quantities reach a stationary state. At the end, we observe that $u(50,50)>u(51,50); v(50,50)>v(51,50); \Delta u(51,50)<0<\Delta u(52,50); \Delta v(50,50)<\Delta v(51,50)<0$. 

\begin{figure}
   \centering

    \begin{tabular}{cc}
        \includegraphics[height=2.2cm,width=5cm]{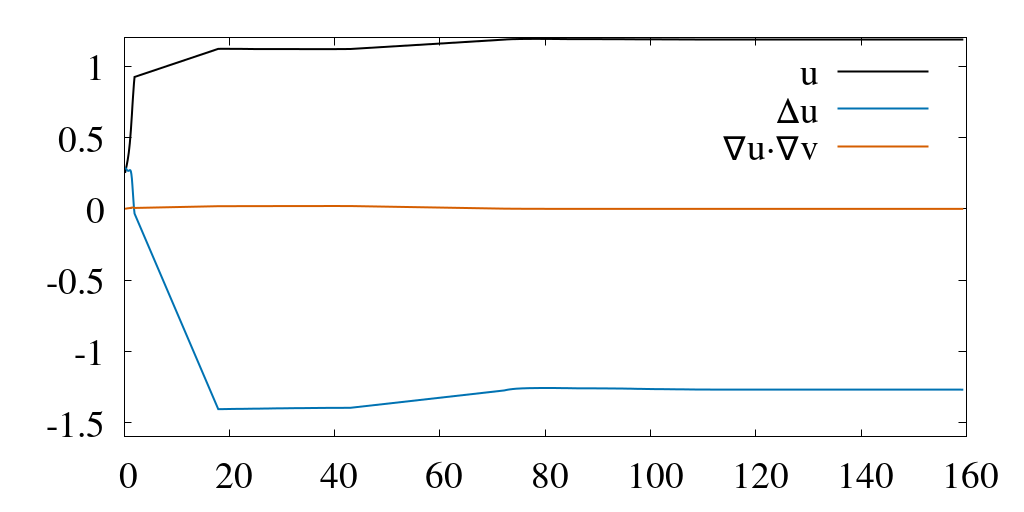} &
        \includegraphics[height=2.2cm,width=5cm]{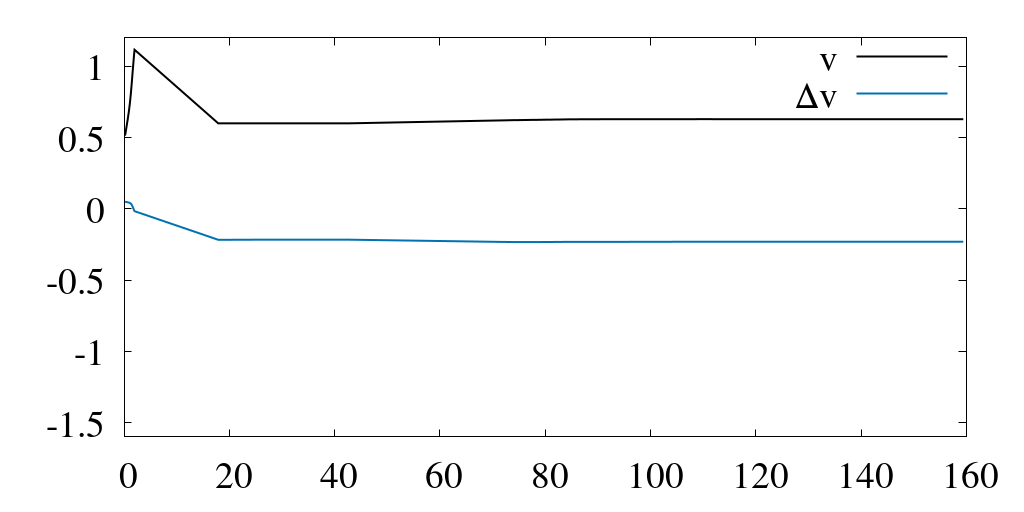} \\

        \includegraphics[height=2.2cm,width=5cm]{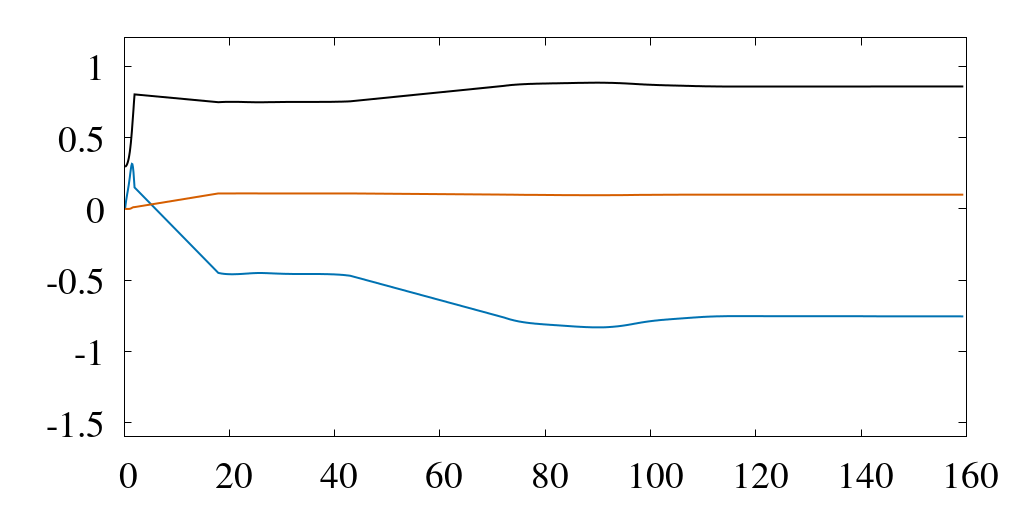} &
        \includegraphics[height=2.2cm,width=5cm]{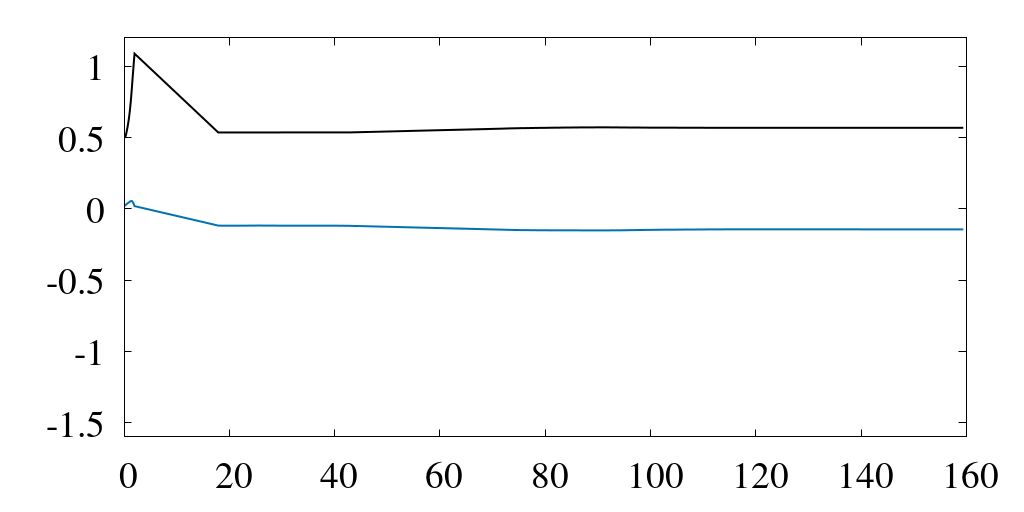} \\

        \includegraphics[height=2.2cm,width=5cm]{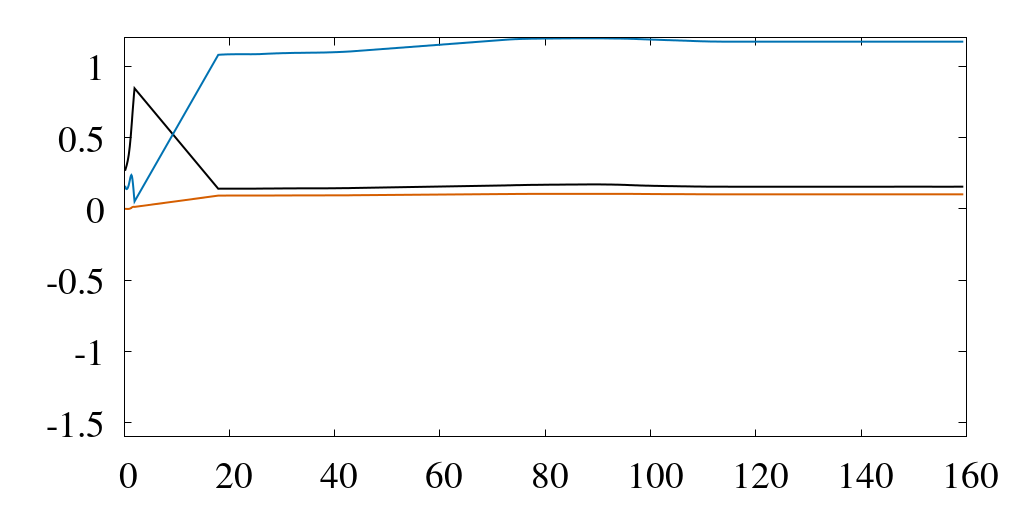} &
        \includegraphics[height=2.2cm,width=5cm]{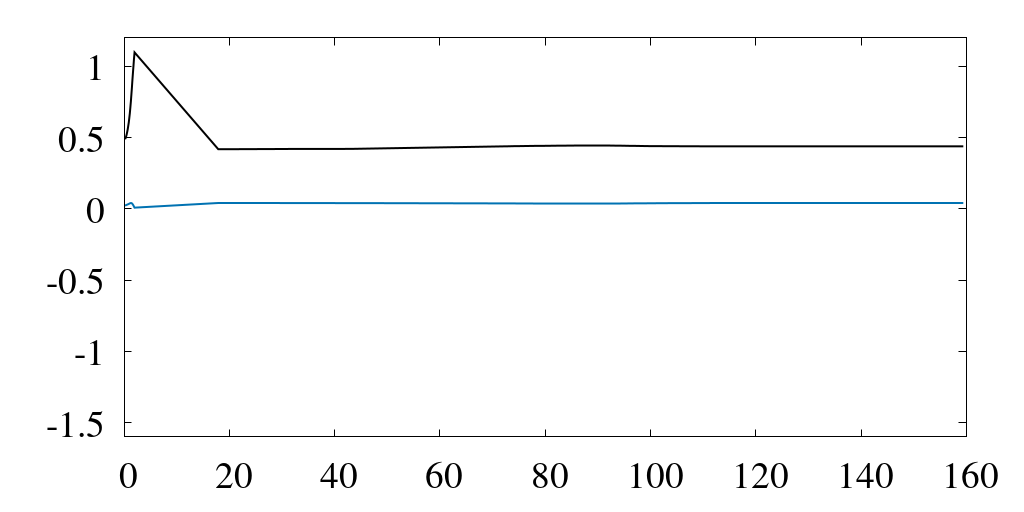} \\

        \includegraphics[height=2.2cm,width=5cm]{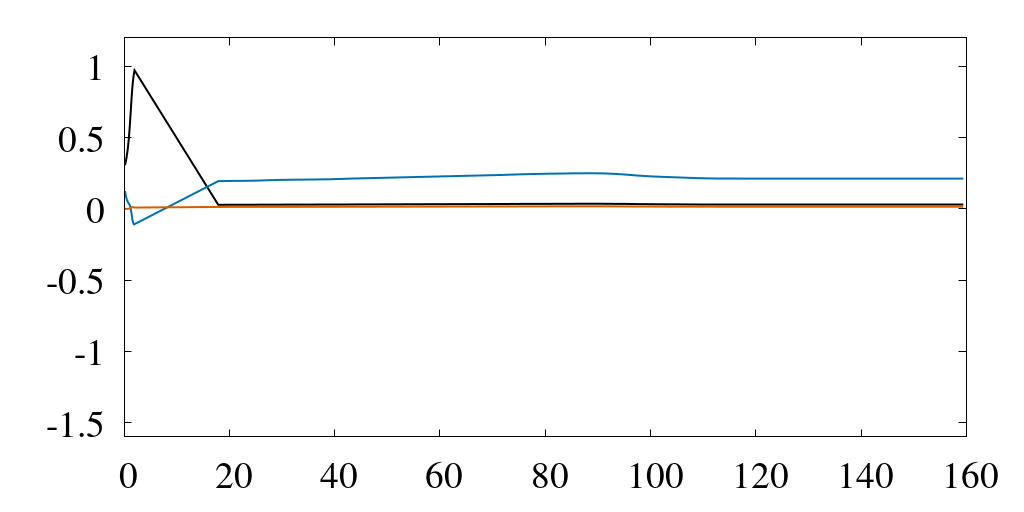} &
        \includegraphics[height=2.2cm,width=5cm]{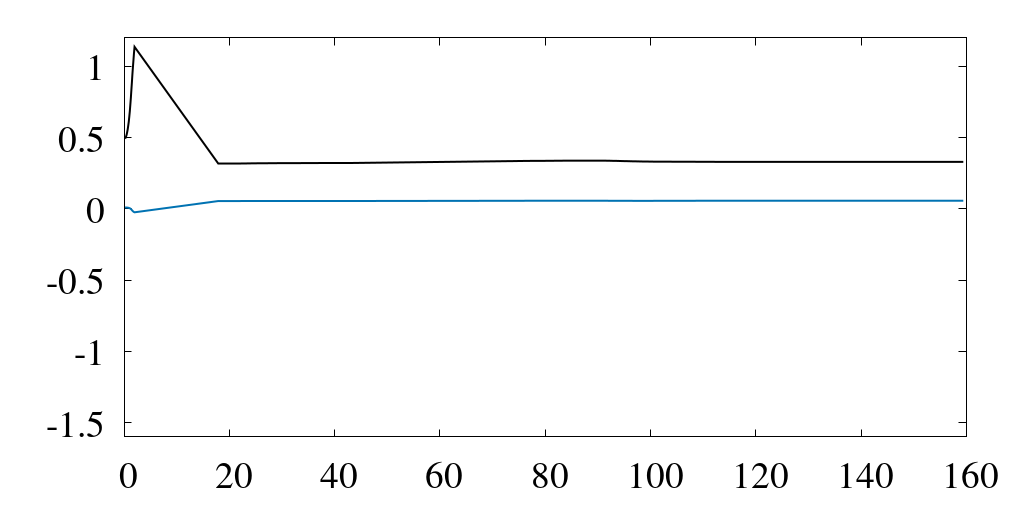} \\

        \includegraphics[height=2.2cm,width=5cm]{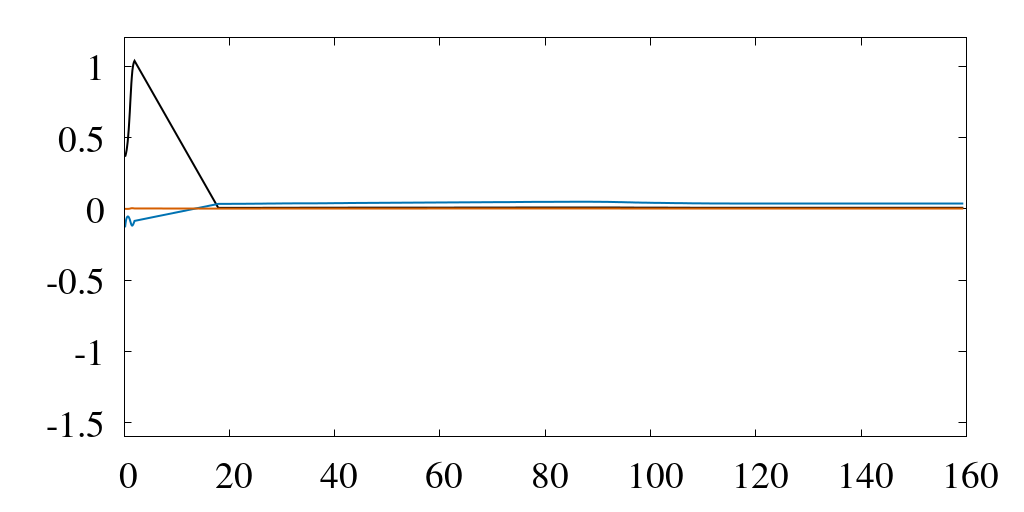} &
        \includegraphics[height=2.2cm,width=5cm]{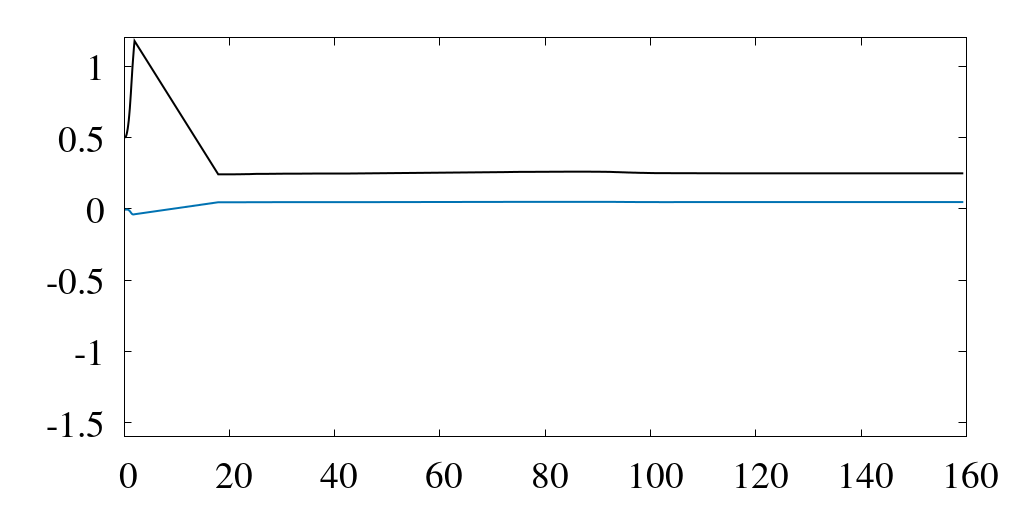}
    \end{tabular}
   % \vspace{2mm}
\makebox[0.48\textwidth]{\small t}
\makebox[0.48\textwidth]{\small t}
    \caption{This figure illustrates the time evolution of the solution of the discretized version of the modified KS model \eqref{eq:mKS} with $\gamma=0.25$ at several consecutive fixed spatial positions. The panels on the left show the time evolution of $u$ (black), $\Delta u$ (blue), and $\nabla u \cdot \nabla v$ (orange), while the panels on the right display the evolution of $v$ (black) and $\Delta v$ (blue). A clear temporal structure emerges, particularly in the time course of $v$, with slopes indicating three distinct time scales that may be classified as very fast, fast, and slow. The first row corresponds to the position $(x,y) = (50,50)$. Comparison with the other rows shows that $u$ and $v$ asymptotically approach a local maximum stationary value, while $\Delta u$ and $\Delta v$ converge to a negative stationary value. The second row corresponds to $(x,y) = (51,50)$, where $u$ and $v$ reach a relatively high stationary value compared to neighboring positions, as confirmed by the negative asymptotic values of $\Delta u$ and $\Delta v$. A similar qualitative behavior, albeit with different asymptotic states (e.g., minima instead of maxima), is observed at $(x,y) = (52,50)$ and $(x,y) = (53,50)$.}
    \label{fig:uvtgam025}
\end{figure}

\section{On a $2\times2$ analog coupled model}
With the goal to investigate some mathematical mechanisms at play for a simpler and more tractable model, we propose the following ODE,

\begin{equation}
\label{eq:2times2}
\left \{
\begin{array}{rl}
      u_{1t}=&  f(u_1)-bu_1(v_2-v_1)+d_u(u_2-u_1) \\
      v_{1t}=& cu_1-ev_1+d_v(v_2-v_1)\\
         u_{2t}=&  f(u_2)-bu_2(v_1-v_2)+d_u(u_1-u_2) \\
      v_{2t}=& cu_2-ev_2+d_v(v_1-v_2)
 \end{array}
 \right.
\end{equation}
 The value of parameters are as follows:
\begin{equation}
         a=7, \, d_u=1,\,   c=3, \, e=2,\, d_v=10, \, \gamma=0.25,\, b\in [5,35]
    \label{eq:param-EDO4v}
\end{equation}
\subsection{General Description}
\Cref{eq:2times2} is a system of two coupled two-dimensional
ODEs designed to mimic some of the mechanisms present in
\Cref{eq:mKSbis}. In particular, both models involve the same
nonlinear growth factor. The terms \(u_2-u_1\) and \(v_2-v_1\)
represent discrete diffusion between the two compartments,
whereas the nonlinear term \(u_1(v_2-v_1)\) is a discrete
analogue of the chemotactic flux \(u\nabla v\). More generally,
terms such as \(u_2-u_1\) constitute a standard linear coupling
between two ODE subsystems.

For suitable parameter values, \Cref{eq:2times2} admits four
coexisting locally asymptotically stable equilibria:
\[
U_1^*=(0,0,0,0),
\qquad
U_3^*=
\left(
1,\frac{c}{e},1,\frac{c}{e}
\right),
\]
and
\[
U_8^*
=
\left(
u_1^{8*},v_1^{8*},u_2^{8*},v_2^{8*}
\right),
\qquad
U_9^*
=
\left(
u_1^{9*},v_1^{9*},u_2^{9*},v_2^{9*}
\right).
\]
The two asymmetric equilibria \(U_8^*\) and \(U_9^*\) are
related by
\[
u_1^{9*}=u_2^{8*},
\qquad
u_2^{9*}=u_1^{8*},
\qquad
v_1^{9*}=v_2^{8*},
\qquad
v_2^{9*}=v_1^{8*},
\]
and satisfy
\[
\gamma<u_2^{8*}<1<u_1^{8*}.
\]

These equilibria represent, respectively, the absence of cells,
high cell concentrations in both compartments, and the
coexistence of high and low cell concentrations in different
compartments. The existence of stable asymmetric equilibria
satisfying \(u_1\neq u_2\) is particularly relevant, as it
provides a minimal-dimensional analogue of the spatial patterns
observed in the PDE model.

The system possesses an exchange symmetry map $P$: if
\[
\bigl(u_1(t),v_1(t),u_2(t),v_2(t)\bigr)
\]
is a solution, then
\[
P\bigl(u_1(t),v_1(t),u_2(t),v_2(t)
\bigr)
=
\bigl(
u_2(t),v_2(t),u_1(t),v_1(t)
\bigr)
\]
is also a solution. In particular, the asymmetric equilibria
\(U_8^*\) and \(U_9^*\) are exchanged by this symmetry.

We begin with a proposition showing that the positive cone is
forward invariant and that all trajectories starting within it
eventually enter a bounded absorbing set.
\begin{prop}
The positive cone
\[
G:=\left\{(u_1,v_1,u_2,v_2)\in\mathbb{R}^4:
u_1>0,\ v_1>0,\ u_2>0,\ v_2>0\right\}
\]
is positively invariant under the flow generated by
\Cref{eq:2times2}. Moreover, the system is dissipative: there
exists a constant \(C>0\) such that, for every bounded set
\(B\subset G\), there exists a time \(T_B>0\) for which
\[
U(0)\in B
\quad\Longrightarrow\quad
\lVert U(t)\rVert\leq C
\qquad\text{for all }t\geq T_B.
\]
Hence, the set
\[
\mathcal B_C:=\left\{U\in G:\lVert U\rVert\leq C\right\}
\]
is a bounded absorbing set for \Cref{eq:2times2}.

\end{prop}
\begin{proof}
First, observe that the set
\[
\left\{
(u_1,v_1,u_2,v_2)\in\mathbb{R}^4
:\ u_1=u_2=0
\right\}
\]
is an invariant manifold. Moreover, on the boundary component
\(u_1=0\), with \(u_2>0\), the vector field satisfies
\[
u_1'>0.
\]
An analogous inward-pointing condition holds on each of the
other boundary components. Consequently, solutions starting in
\(G\) cannot cross its boundary, and \(G\) is positively
invariant.

To prove the existence of a bounded absorbing set, we introduce
the Lyapunov-type energy functional
\[
V(t)
=
u_1(t)^2+u_2(t)^2
+\mu\left(v_1(t)^2+v_2(t)^2\right),
\]
where \(\mu>0\) is chosen appropriately. The remainder of the
proof consists in deriving a differential inequality for \(V\)
that yields a uniform ultimate bound on the solutions.
\end{proof}

The next proposition characterizes the stationary solutions. We assume that the denominators do not vanish.
\begin{prop}
 The stationary solutions of \Cref{eq:2times2} satisfy the following equations. 
\begin{equation}
\label{eq:sta2times2}   
\left\{\begin{array}{rl}
     u_2=&\phi(u_1) \\
    u_1=&\phi(u_2) \\
    v_1=&\alpha_1 u_1 +\alpha_2 u_2\\
    v_2=&\alpha_1 u_2 +\alpha_2 u_1
\end{array}\right.
\end{equation}
with
\begin{equation}
\label{eq:phi}   
    \phi(u)=u+\frac{f(u)}{\dfrac{bc}{e+2d_v}u-d_u}.
\end{equation}
and
\begin{equation}
\label{eq:paramalpha}   
     \alpha_1=\frac{c(d_v+e)}{(d_v+e)^2-d_v^2},  \,\,    \alpha_2=\frac{cd_v}{(d_v+e)^2-d_v^2}. 
\end{equation}
Equivalently, stationary solutions of \Cref{eq:2times2} satisfy the fixed point equation
\begin{equation}
    \label{eq:PF}
    u_1=\phi\circ\phi(u_1)
\end{equation}
along with $v_1,u_2,v_2$ solutions of \Cref{eq:sta2times2}. 
\end{prop}
\begin{proof}
    The proposition follows from elementary computations.
\end{proof}
We note that the geometric characterization of the stationary
points provided by \Cref{eq:sta2times2} is particularly useful
for tracking the equilibria as the parameter \(b\) varies.
Indeed, in the \((u_1,u_2)\)-plane, the stationary solutions
correspond to the intersections of the curves
\[
u_2=\phi(u_1)
\qquad\text{and}\qquad
u_1=\phi(u_2).
\]
In the following paragraph, we describe the behavior observed
numerically for selected relevant values of \(b\); see also
\Cref{fig:EDO4v-stapoints,fig:EDO4v}.

\textbf{Dynamical description}
For parameters as in  \Cref{eq:param-EDO4v}, with $b\in (10,25)$ after numerical investigation, the dynamics can be summarized as follows: for $b$ close to $10$, there is only three stationary points $U^{1*}=(0,0,0,0)$, $U^{2*}=(\gamma,\frac{c}{e}\gamma,\gamma,\frac{c}{e}\gamma)$ and $U^{3*}=(1,\frac{c}{e},1,\frac{c}{e})$. $U^{1*}$ and $U^{3*}$ are locally asymptotically stable, $U^{2*}$ is unstable. As $b$ increases, two new unstable stationary points (with $u_1\neq u_2$), $U^{4*},U^{5*}$ appear. As $b$ increases more, four new stationary points appear: two unstable  $U^{6*},U^{7*}$ and two locally asymptotically stable  $U^{8*},U^{9*}$. We refer to \Cref{fig:EDO4v-stapoints} and \Cref{fig:EDO4v} for illustrations. Finally, we note that stationary solutions $U^{8*},U^{9*}$ provide illustrations for the effects of nonzero growth factors, diffusive factors and chemotactic factors at equilibrium. Indeed, at $U^{8*}$, we have a negative growth factor $f(u_1^{8*})<0$ and a negative diffusive factor $u_2^{8*}-u_1^{8*}<0$ (local maximum, negative diffusion),  compensated by a positive chemotactic factor $-bu_1(v_2-v_1)>0$; a positive growth factor $f(u_2^{8*})>0$ and a positive diffusive factor $u_1^{8*}-u_2^{8*}<0$ (local minimum, positive diffusion),  compensated by a negative chemotactic factor $-bu_2(v_1-v_2)<0$. We recover in this low dimensional system the principles that led to the formation of BMEC patterns in the KS PDE. It follows that \Cref{eq:2times2} represents a low-dimensional  prototype for non homogeneous concentrations states featuring growth, diffusive and chemotactic factors. In the next section, we provide a more complete bifurcation analysis.

\begin{figure}
    \begin{center}

    \includegraphics[height=4.3cm,width=4.3cm]{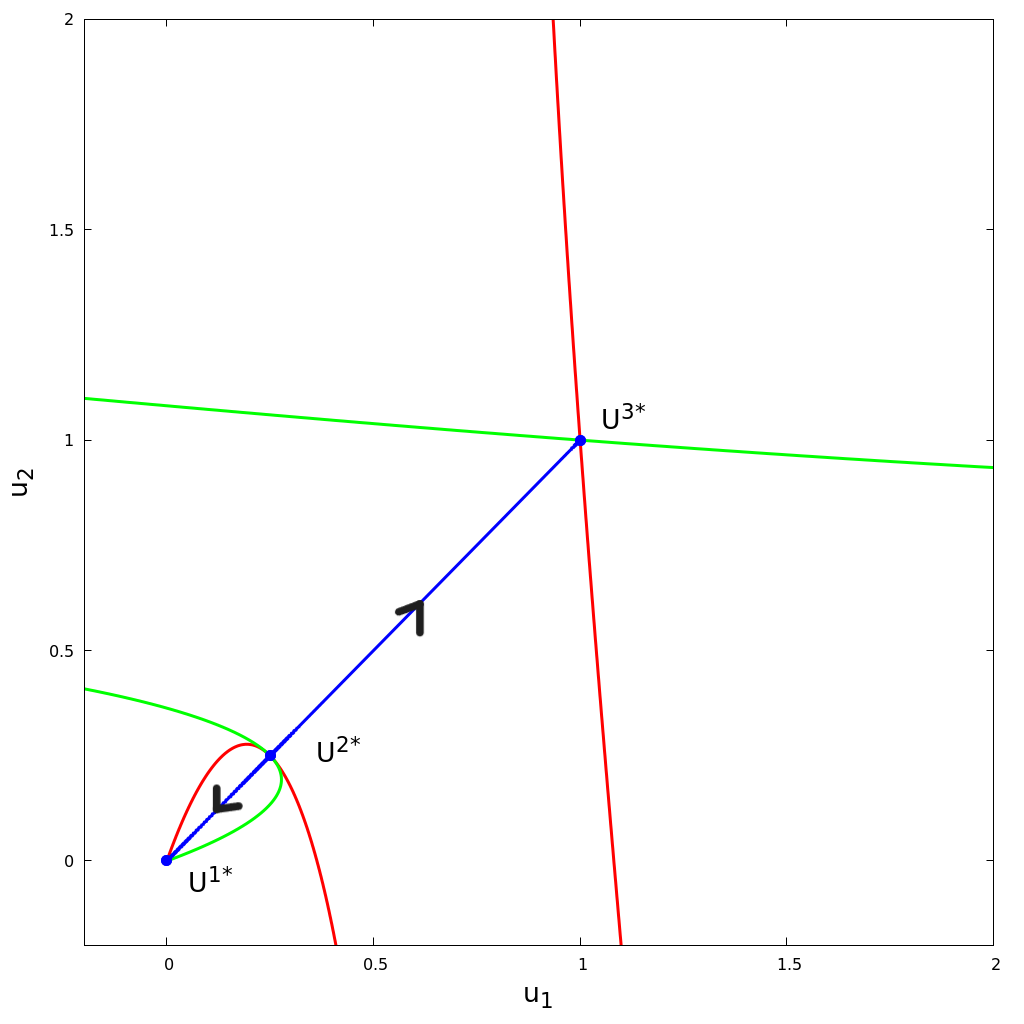}
      \includegraphics[height=4.3cm,width=4.3cm]{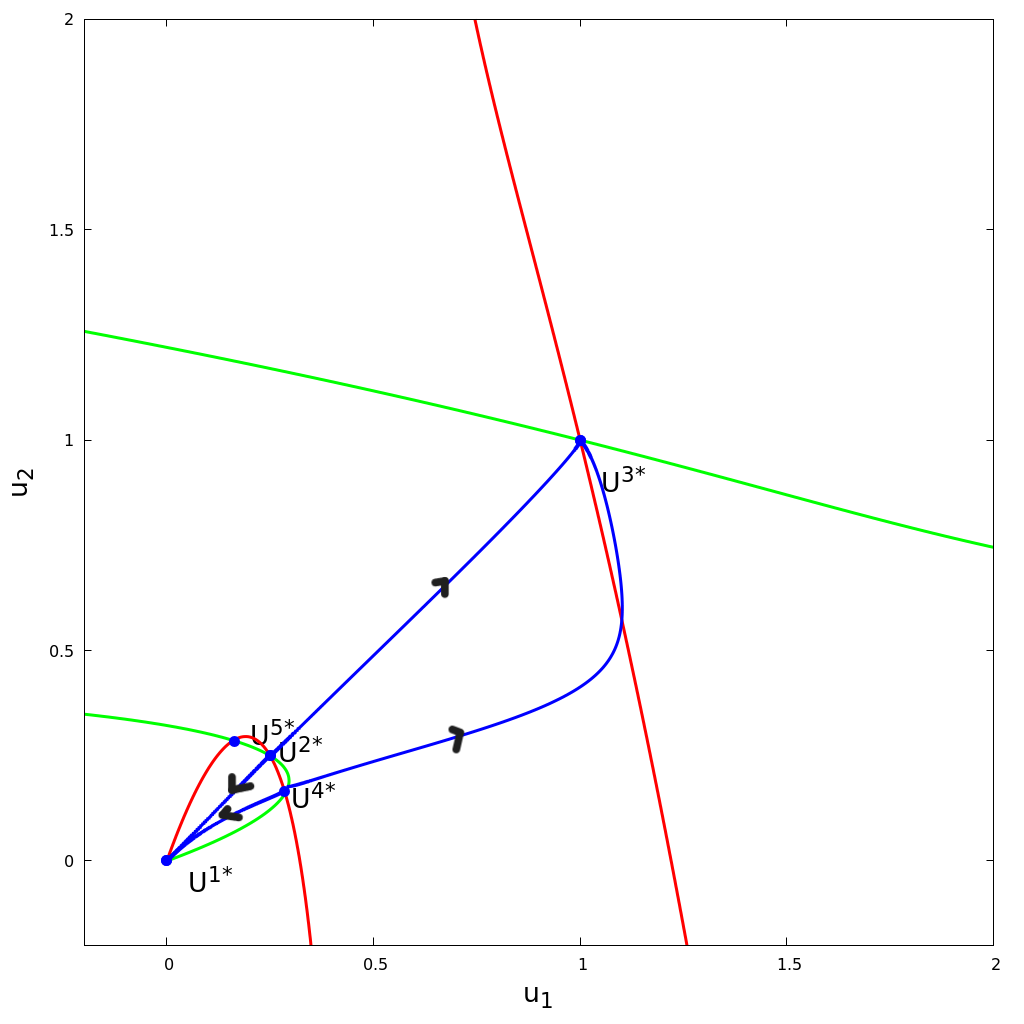}
       \includegraphics[height=4.3cm,width=4.3cm]{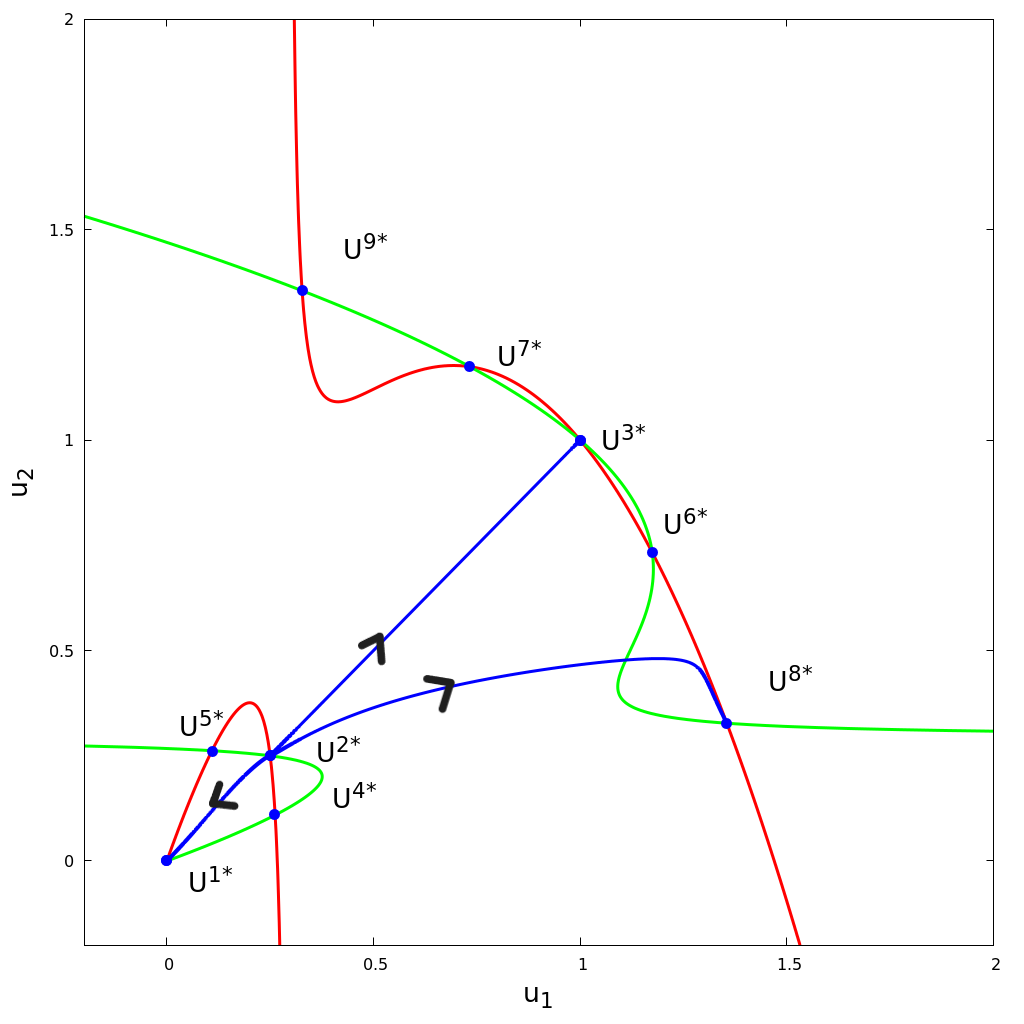}
 \end{center}
    \caption{Stationary solutions and some trajectories of \Cref{eq:2times2} projected in the $u_1,u_2$ phase plan. Left panel, $b=10$, middle panel, $b=15$, right panel, $b=25$. The red curves represent $u_2=\phi(u_1)$. Green curves represent $u_1=\phi(u_2)$. Blue dots are projections of stationary solutions at the intersection of red and green curves.  As the parameter $b$ increases from $10$ to $25$, new stationary points appear. On the left side, for $b=10$, we observe only the three stationary solutions where $u_1=u_2$ and $v_1=v_2$ which are known from the study of \Cref{eq:mKSODE}. At $b=15$, we observe two unstable stationary solutions $U^{4*}$ and $U^{5*}$. At $b=25$, we observe four more stationary solutions $U^{6*}$, $U^{7*}$, which are unstable and $U^{8*}$, $U^{9*}$, which are locally asymptotically stable. Blue curves represent the projections of some trajectories of \Cref{eq:2times2} projected in the $u_1,u_2$ phase plan. They illustrate heteroclinic orbits between unstable and stable stationary points. Note that there is a symmetry in the system.    }
    \label{fig:EDO4v-stapoints}
\end{figure}

\begin{figure}
    \begin{center}
    \includegraphics[height=4.3cm,width=4.3cm]{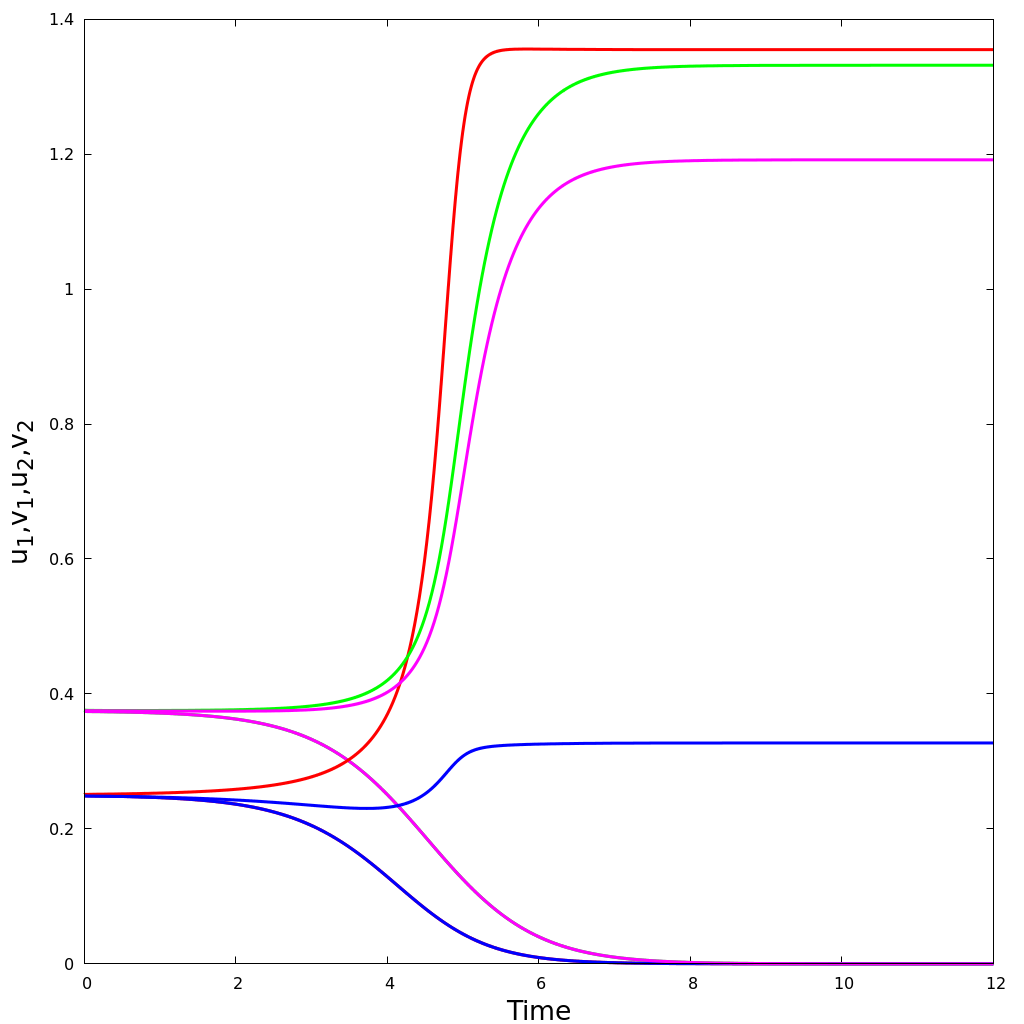}
    \includegraphics[height=6cm,width=6cm]{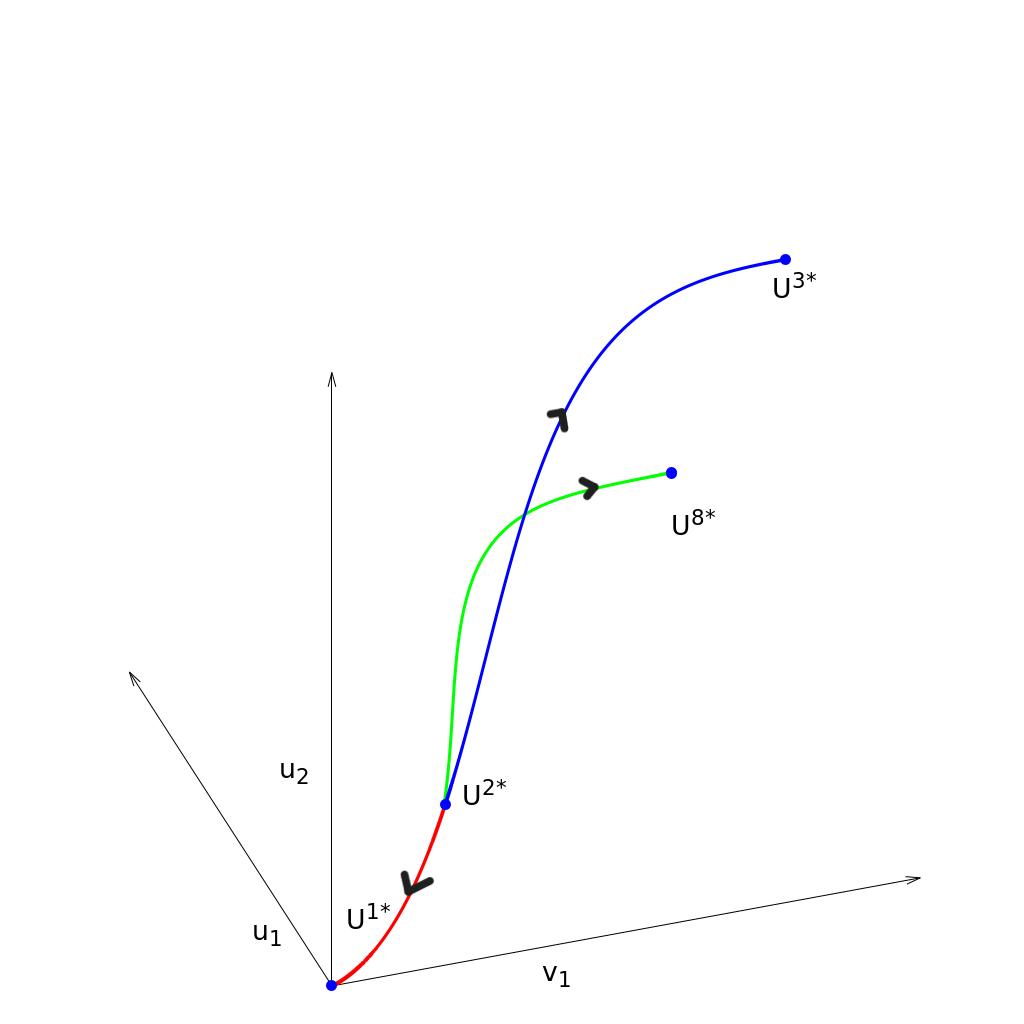}
    
 \end{center}
    \caption{Solutions of \Cref{eq:2times2} for $b=25$. The left panel illustrates $u_1$ (red), $v_1$ (green), $u_2$ (blue), $v_2$ (magenta) as functions of time, for two initial conditions close to the stationary point $U^{2*}=(\gamma,\frac{c}{e}\gamma,\gamma,\frac{c}{e}\gamma)$. We observe that one of the solutions converges toward $U^{8*}=(u_1^{8*},v_1^{8*},u_2^{8*},v_2^{8*})$ (with
$\gamma<u_2^{8*}<1<u_1^{8*}$) while the other one converges  toward $U^{1*}=(0,0,0,0)$. The right panel illustrates three heteroclinic orbits going from $U^{2*}$ toward $U^{1*}$ (red), $U^{3*}$ (blue) and $U^{8*}$ (green) as $t$ goes from $-\infty$ to $+\infty$.  }
    \label{fig:EDO4v}
\end{figure}
\subsection{Bifurcation Analysis}
\textbf{Bifurcation Diagram} 
In this subsection, we further investigate the system through a bifurcation analysis. The bifurcation diagram corresponding to the parameter values specified in \Cref{eq:param-EDO4v} is shown in \Cref{fig:BifDiag} and described below. A theoretical formulation and proofs of the results at the bifurcation points are provided in the next paragraph.

For $b\in(5,b_1)$, where
\[b_1=\frac{(e+2d_v)\bigl(2d_u-f'(\gamma)\bigr)}
     {2c\gamma}\simeq10.08,\]
     the coupled system \cref{eq:2times2} admits only the three homogeneous stationary solutions satisfying $(u_1,v_1)=(u_2,v_2)$ with $u_1\in \{0,\gamma,1\}$, $v_1=\frac{c}{e}u_1$. The stationary points $(0,0)$ and $(1,\frac{c}{e})$ are locally asymptotically stable whereas $(\gamma,\frac{c}{e}\gamma)$ is unstable. At $b=b_1$, the stationary solution $(\gamma,\frac{c}{e}\gamma)$ undergoes a supercritical pitchfork bifurcation, from which two new unstable solutions emerge. For $b\in(b_1,b_2)$, where $b_2\simeq 23.14$ is the solution of a nonlinear equation discussed in the next paragraph, the system retains the same structure, with five stationary solutions and unchanged stability properties. At $b=b_2$, a saddle-node bifurcation occurs. Four new nonhomogeneous stationary points—two locally stable and two unstable—appear for $b>b_2$. Equivalently, as $b$ crosses $b_2$ from right to left, the stable and unstable stationary points coalesce and disappear. Owing to the invariance of the system under the permutation $P$, four stationary points simultaneously disappear when $b$ is decreased through $b_2$, or equivalently appear when $b$ is increased through $b_2$.
     Finally, at $b=b_3$ where 
\[b_3=\frac{(e+2d_v)\bigl(2d_u-f'(1)\bigr)}
     {2c}\simeq 26.52,\]
a subcritical pitchfork bifurcation occurs: as $b$ crosses $b_3$ from right to left, two unstable stationary points emerge from the homogeneous stationary point $(u_1,v_1,u_2,v_2)=(1,\frac{c}{e},1,\frac{c}{e}$.
\begin{figure}
    \begin{center}
    \includegraphics[height=5cm,width=8cm]{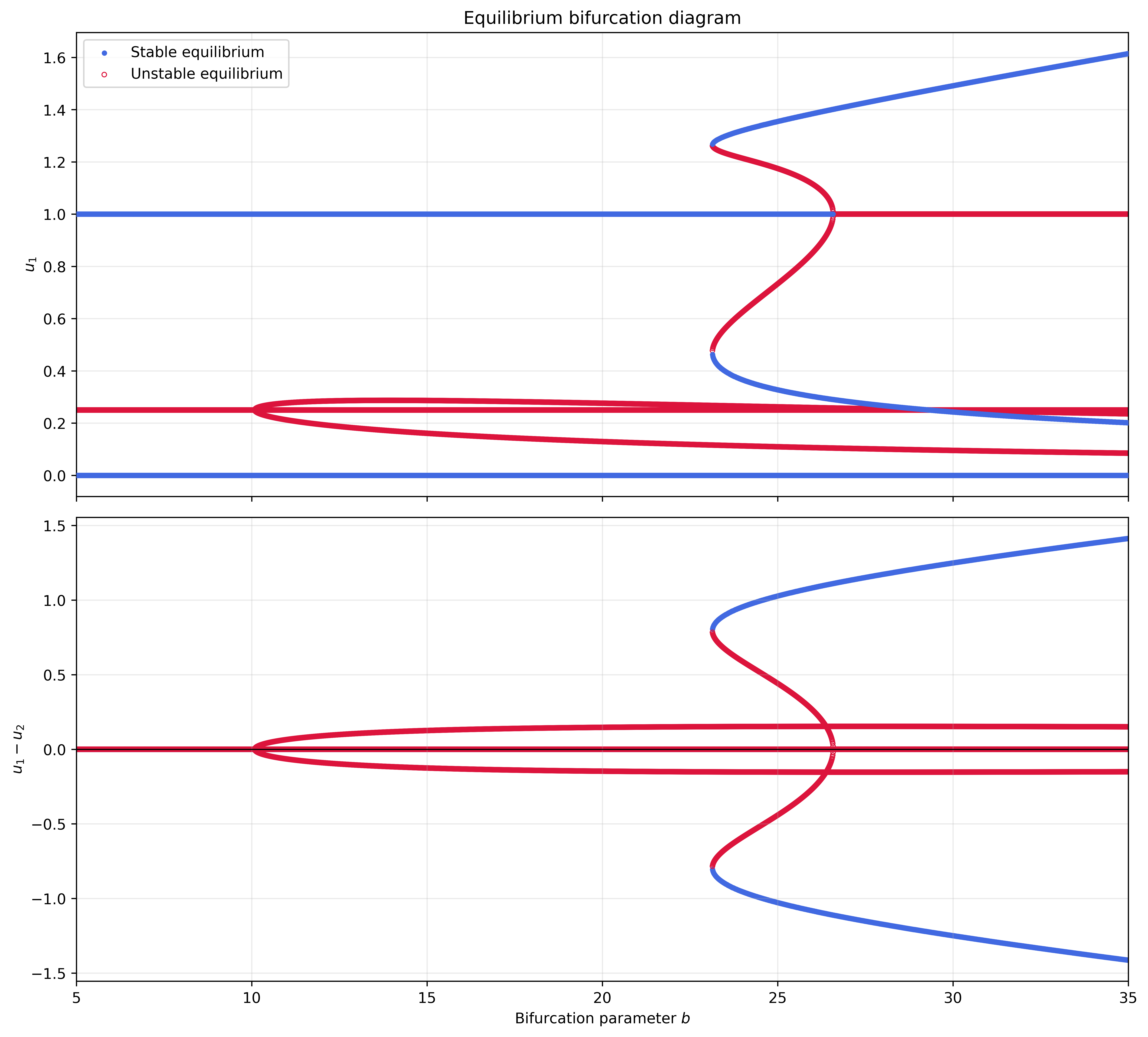}

 \end{center}
    \caption{\textbf{Bifurcation diagram of the stationary solutions of system
\cref{eq:2times2} for the parameter values given in
\cref{eq:param-EDO4v}.}
The top panel shows $u_1$ as a function of the bifurcation parameter $b$,
whereas the bottom panel shows the symmetry-breaking measure $u_1-u_2$.
The critical values $b_1\simeq 10.08$, $b_2\simeq 23.14$, and
$b_3\simeq 26.52$ correspond, respectively, to a supercritical pitchfork
bifurcation, a saddle-node bifurcation, and a subcritical pitchfork
bifurcation. Homogeneous stationary solutions satisfy $u_1-u_2=0$,
whereas nonzero values of $u_1-u_2$ correspond to nonhomogeneous
stationary solutions. Stable branches are shown in blue, and unstable
branches are shown in red.  }
    \label{fig:BifDiag}
\end{figure}

\textbf{Pitchfork and Saddle-Node Bifurcations}
In this paragraph, we further investigate the bifurcation points described above by providing additional theoretical results and explicit numerical computations. We recall, see for example \cite{kuznetsov2004elements}, that the normal form of a pitchfork bifurcation is
\[
\dot{x}=\mu x+\alpha x^3,
\]
where $\mu$ is the
bifurcation parameter, and $\alpha\neq 0$. The trivial branch $x=0$
exists for all $\mu$, while the two symmetry-related nontrivial branches
are given by
\[
x=\pm\sqrt{-\frac{\mu}{\alpha}},
\]
whenever $-\mu/\alpha>0$. For $\alpha<0$, the bifurcation is
\emph{supercritical}: two stable branches emerge as the trivial branch
loses stability. For $\alpha>0$, the bifurcation is \emph{subcritical}:
two unstable branches coexist with the stable trivial branch before the
bifurcation point. The normal form of a saddle-node (fold) bifurcation is
\[
\dot{x} = \mu - x^2,
\]
The equilibria satisfy
\[
\mu - x^2 = 0,
\]
so
\[
x^* = \pm \sqrt{\mu}.
\]

Therefore:
\begin{itemize}
    \item If \(\mu < 0\), there are no equilibria.
    \item If \(\mu = 0\), there is one non-hyperbolic equilibrium at \(x=0\).
    \item If \(\mu > 0\), there are two equilibria: one stable and one unstable.
\end{itemize}

Thus, at \(\mu=0\), a stable and an unstable equilibrium collide and
annihilate, or are created, producing the characteristic fold.
The following theorem provides the theoretical result for the
pitchfork bifurcation appearing in \Cref{eq:2times2}. We state it for the general set of positive
parameters,
\begin{equation}
    a,b,c,e,d_u,d_v>0.
    \label{eq:positiveparam2times2}
\end{equation}

Let \(u\in\{\gamma,1\}\), and define the spatially homogeneous
equilibrium
\[
\mathcal E_u
=
\left(
u,\frac{cu}{e},
u,\frac{cu}{e}
\right).
\]
Assume that
\begin{equation}
    2d_u-f'(u)>0,
    \label{eq:positivebforpitchfork}
\end{equation}
and define
\begin{equation}
    b_*= \frac{\big( e+2d_v \big)\bigl(2d_u-a f'(u)\bigr)}
     {2cu}.
     \label{eq:b*}
\end{equation}
Finally, assume that
\begin{equation}
Q
=
-a
- \frac{
\left(
\frac{1}{2}f''(u)
+ 2b_*\frac{c}{e+2d_v}
\right)
\left(
f''(u)
+ 2b_*\frac{c}{e+2d_v}
\right)
}{
f'(u)
}
\neq 0.
\label{eq:Q}
\end{equation}
\begin{thm}[Symmetry-breaking pitchfork bifurcation]
Under assumptions \eqref{eq:positiveparam2times2},
\eqref{eq:positivebforpitchfork}, and \eqref{eq:Q},
\Cref{eq:2times2} undergoes a nondegenerate symmetry-breaking
pitchfork bifurcation at
\[
(\mathcal E_u,b_*)
=
\left(
u,\frac{cu}{e},
u,\frac{cu}{e},
b_*
\right).
\]

More precisely, after a smooth change of coordinates and a
positive rescaling of time, the reduced equation on the
one-dimensional center manifold is
\[
\dot \xi
=
\xi\left[
A(b-b_*)+Q\xi^2
+ \mathcal O\left(
(b-b_*)^2
+ |b-b_*|\xi^2
+ \xi^4
\right)
\right],
\]
where
\[
A=\frac{2cu}{e+2d_v}>0.
\]

Consequently, in addition to the homogeneous branch \(\xi=0\),
there are two symmetry-related nonhomogeneous equilibrium
branches satisfying
\[
\xi_\pm(b)
=
\pm
\sqrt{-\frac{A(b-b_*)}{Q}}
+ \mathcal O(|b-b_*|),
\]
whenever
\[
-\frac{A(b-b_*)}{Q}>0.
\]

In particular:
\begin{enumerate}
    \item If \(Q<0\), the two nonhomogeneous branches exist for
    \(b>b_*\), and the pitchfork is supercritical.

    \item If \(Q>0\), the two nonhomogeneous branches exist for
    \(b<b_*\), and the pitchfork is subcritical.
\end{enumerate}
\end{thm}
\begin{proof}
We introduce the variables
\begin{equation}
 x=\frac{u_1-u_2}{2},
\qquad
y=\frac{v_1-v_2}{2},
\qquad
s=\frac{u_1+u_2}{2},
\qquad
r=\frac{v_1+v_2}{2}.
\label{eq:chvarxysr}
\end{equation}

The symmetry exchanging the two components,
\[
P:(u_1,v_1,u_2,v_2)
\longmapsto
(u_2,v_2,u_1,v_1),
\]
becomes
\[
R:(x,y,s,r)\longmapsto(-x,-y,s,r).
\]

In the new coordinates, \Cref{eq:2times2} becomes

\begin{equation}
\left\{
\begin{array}{rcl}
\dot x
&=&
\frac{1}{2}\bigl(f(s+x)-f(s-x)\bigr)
+2bsy-2d_ux,\\[2mm]
\dot y
&=&
cx-(e+2d_v)y,\\[1mm]
\displaystyle
\dot s
&=&
\bigl(f(s+x)+f(s-x)\bigr)
+2bxy,\\[2mm]
\dot r
&=&
cs-er.
\end{array}
\right.
\label{eq:2times2xysr}
\end{equation}

We rerwite \Cref{eq:2times2xysr} equivalently  as

\begin{equation}
\left\{
\begin{array}{rcl}

\dot x
&=&
\bigl(f'(s)-2d_u\bigr)x
-a x^3+2bsy,\\[1mm]
\dot y
&=&
cx-(e+2d_v)y,\\[1mm]
\dot s
&=&
f(s)+\dfrac{1}{2}f''(s)x^2+2bxy,\\[2mm]
\dot r
&=&
cs-er.
\end{array}
\right.
\label{eq:2times2xysr_eq}
\end{equation}

Now, the homogeneous equilibrium are
\[
x=y=0,\quad f(s)=0, \quad s=u, \quad r=\frac{u}{e}.
\]
with
\[u\in\{0,\gamma,1\},
\]

The Jacobian matrix around these stationary points reads

\begin{equation}
    J(u,b)
=
\begin{pmatrix}
 f'(u)-2d_u & 2bu & 0 & 0\\
c & -(e+2d_v) & 0 & 0\\
0 & 0 &f'(u) & 0\\
0 & 0 & c & -e
\end{pmatrix}.
\label{eq:Jac}
\end{equation}
The eigenvalues of the Jacobian are therefore
\[\lambda_3=f'(u), \lambda_4=-e\]
and $\lambda_1,\lambda_2$, which are the roots of the quadratic equation
\[\big(f'(u)-2d_u-\lambda\big)\big(-(e+2d_v)-\lambda\big)-2buc=0.\]
The stationary point
\[
(x,y,s,r)=(0,0,0,0)
\]
is always locally asymptotically stable. The condition for a zero
eigenvalue at \(u=1\) or \(u=\gamma\) is
\begin{equation}
  b=\frac{(2d_u-f'(u))(e+2d_v)}{2uc}  
\end{equation}
which gives the expression of $b_*(u)$. For $b=b_*(u)$, we have therefore,
\[\lambda_1=f'(u)-2d_u-(e+2d_v)<0, \lambda_2=0.\]

As \(b\) crosses \(b_*\) from left to right, \(\lambda_2\)
crosses zero from negative to positive.

At \(u=\gamma\), we have
\[
\lambda_1<0,
\qquad
\lambda_3>0,
\qquad
\lambda_4<0,
\]
whereas at \(u=1\), we have
\[
\lambda_1<0,
\qquad
\lambda_3<0,
\qquad
\lambda_4<0.
\]
This means that the stationary point
\[
\left(0,0,\gamma,\frac{c\gamma}{e}\right)
\]
is always unstable because \(\lambda_3>0\), whereas the
stationary point
\[
\left(0,0,1,\frac{c}{e}\right)
\]
is locally asymptotically stable for \(b<b_*(1)\) and unstable
for \(b>b_*(1)\).

To characterize the nonhomogeneous stationary branches emerging
from \(\mathcal E_u\), with \(u\in\{\gamma,1\}\), we consider the
stationary equations associated with
\Cref{eq:2times2xysr_eq}:

\begin{equation}
\left\{
\begin{array}{rcl}
0
&=&
\bigl(f'(s)-2d_u\bigr)x
-a x^3+2bsy,\\[1mm]
0
&=&
cx-(e+2d_v)y,\\[1mm]
0
&=&
f(s)+\dfrac{1}{2}f''(s)x^2+2bxy,\\[2mm]
0
&=&
cs-er.
\end{array}
\right.
\label{eq:2times2xysr_stationary}
\end{equation}
This rewrites,
\begin{equation}
\left\{
\begin{array}{rcl}
0
&=&
\bigl(f'(s)-2d_u +2bs\frac{c}{e+2d_v}\bigr)x
-a x^3,\\[1mm]
y
&=&
\frac{c}{e+2d_v}x,\\[1mm]
0
&=&
f(s)+\big(\dfrac{1}{2}f''(s)+2b\frac{c}{e+2d_v}\big)x^2,\\[2mm]
 r
&=&
\frac{c}{e}s.
\end{array}
\right.
\label{eq:2times2xysr_stationary2}
\end{equation}
Since
\[f'(s)\neq 0\] at $s=\gamma$ or $s=1$, we can apply the implicit function theorem to write
\[s=s(x^2)=u+\alpha x^2+ \mathcal{O}(x^4)\]
in a neighborhood of \(x=0\), with
\[
0
=
f\bigl(s(x^2)\bigr)
+ \left(
\frac{1}{2}f''\bigl(s(x^2)\bigr)
+ 2b\frac{c}{e+2d_v}
\right)x^2.
\]
We obtain
\[
f(u)
+ \alpha f'(u)x^2
+ \left(
\frac{1}{2}f''(u)
+ 2b\frac{c}{e+2d_v}
\right)x^2
+ \mathcal O(x^4)
=0.
\]
Since \(f(u)=0\), it follows that
\[
\alpha
=
-\frac{
\frac{1}{2}f''(u)
+ 2b\frac{c}{e+2d_v}
}{
f'(u)
}.
\]

Substituting
\[
s=u+\alpha x^2+\mathcal O(x^4)
\]
into the first equation of
\Cref{eq:2times2xysr_stationary2}, we obtain
\[
\left[
f'(u)-2d_u
+ 2bu\frac{c}{e+2d_v}
+ \left(
\alpha
\left(
f''(u)+2b\frac{c}{e+2d_v}
\right)
-a
\right)x^2
+ \mathcal O(x^4)
\right]x
=0.
\]

Introducing
\[
\mu=b-b_*,
\]
the previous equation becomes
\[
\left[
2\mu u\frac{c}{e+2d_v}
+ \left(
\alpha
\left(
f''(u)
+ 2(b_*+\mu)\frac{c}{e+2d_v}
\right)
-a
\right)x^2
+ \mathcal O(x^4)
\right]x
=0.
\]

Let
\[
A=2u\frac{c}{e+2d_v}
\]
and
\[
Q
=
-a
- \frac{
\left(
\frac{1}{2}f''(u)
+ 2b_*\frac{c}{e+2d_v}
\right)
\left(
f''(u)
+ 2b_*\frac{c}{e+2d_v}
\right)
}{
f'(u)
}.
\]
The reduced stationary equation is therefore
\[
x\left[
A\mu+Qx^2
+ \mathcal O\left(
\mu^2+|\mu|x^2+x^4
\right)
\right]
=0.
\]

Since \(Q\neq0\), the implicit function theorem provides two
branches of nonhomogeneous stationary solutions satisfying
\[
x^2
=
-\frac{A}{Q}\mu
+ \mathcal O(\mu^2).
\]
When \(Q<0\), the two branches exist for \(\mu>0\), whereas when
\(Q>0\), the two branches exist for \(\mu<0\).

The theorem now follows from standard results on center manifolds
for systems with symmetry; see
\cite{kuznetsov2004elements,Perko2004}. A standard method for
computing the center manifold and the reduced equation is the
Lyapunov--Schmidt reduction; see
\cite{Golubitsky1985,Vanderbauwhede2011}.

Let \(q\) be the eigenvector associated with the zero eigenvalue,
defined by
\begin{equation}
\label{eq:q-general-u}
q=
\begin{pmatrix}
1\\[1mm]
\dfrac{c}{e+2d_v}\\[1mm]
0\\[1mm]
0
\end{pmatrix},
\end{equation}
and let \(p\) be the corresponding left eigenvector, defined by
\begin{equation}
\label{eq:p-general-u}
p
=
p_0
\begin{pmatrix}
1\\[1mm]
\dfrac{2b_*u}{e+2d_v}\\[1mm]
0\\[1mm]
0
\end{pmatrix},
\qquad
p_0
=
\left(
1+\frac{2b_*uc}{(e+2d_v)^2}
\right)^{-1}.
\end{equation}
Notice that
\[
\langle p,q\rangle=1.
\]

For \(Z\) in a neighborhood of the stationary point, define
\[
Z=\xi q+W,
\]
where
\begin{equation}
\label{eq:center-coordinate-general-u}
\xi
=
\langle p,Z\rangle
=
p_0\left(
x+\frac{2b_*u}{e+2d_v}y
\right).
\end{equation}
Then
\begin{equation}
\label{eq:center-stable-decomposition}
\langle p,W\rangle=0.
\end{equation}

Standard center-manifold computations show that, after a smooth
change of coordinates and a positive rescaling of time, the
reduced equation on the one-dimensional center manifold is
\[
\dot \xi
=
\xi\left[
A(b-b_*)+Q\xi^2
+ \mathcal O\left(
(b-b_*)^2
+ |b-b_*|\xi^2
+ \xi^4
\right)
\right].
\]
This completes the proof.
\end{proof}
We now proceed with the analysis of the saddle-node bifurcation.
Let
\[
g_1(z,s,b)
=
f'(s)+\frac{2c}{e+2d_v}bs-2d_u-az,
\]
\[
g_2(z,s,b)
=
f(s)
+ \left(
\frac{1}{2}f''(s)
+ \frac{2c}{e+2d_v}b
\right)z,
\]
and
\[
h(z,s,b)
=
\left(
\partial_z g_1\,\partial_s g_2
- \partial_s g_1\,\partial_z g_2
\right)(z,s,b).
\]

Assume that there exists \((z_*,s_*,b_*)\) such that
\begin{equation}
\label{eq:condFold1}
z_*>0,
\qquad
g_1(z_*,s_*,b_*)=0,
\qquad
g_2(z_*,s_*,b_*)=0,
\qquad
h(z_*,s_*,b_*)=0.
\end{equation}

Define
\[
G
=
\begin{pmatrix}
g_1\\
g_2
\end{pmatrix}
\]
and let
\[
J_*
=
D_{(z,s)}G(z_*,s_*,b_*).
\]
Let \(p\) and \(q\) satisfy
\[
J_*q=0,
\qquad
J_*^{\mathsf T}p=0,
\qquad
\langle p,q\rangle=1.
\]
We write
\[
p=
\begin{pmatrix}
p_1\\
p_2
\end{pmatrix},
\qquad
q=
\begin{pmatrix}
q_z\\
q_s
\end{pmatrix}.
\]

Finally, assume that
\begin{equation}
\label{eq:condFold2}
p_1s_*+p_2z_*\neq0
\end{equation}
and
\begin{equation}
\label{eq:condFold3}
-6ap_1q_s^2
+ p_2\left(
-6aq_sq_z+f''(s_*)q_s^2
\right)
\neq0.
\end{equation}

The variables correspond to those introduced in
\Cref{eq:chvarxysr}. We then have the following result.

\begin{thm}[Saddle-node bifurcation of non-symmetric equilibria]
\label{thm:saddle-node}
Assume that
\Cref{eq:condFold1,eq:condFold2,eq:condFold3} hold. Assume, in
addition, that the full four-dimensional linearization at the
corresponding equilibria has no other eigenvalues on the
imaginary axis. Then system \eqref{eq:2times2} undergoes a
nondegenerate saddle-node bifurcation at \(b=b_*\).

Let
\[
x_*=\sqrt{z_*}.
\]
At \(b=b_*\), the system has two symmetry-related non-symmetric
equilibria:
\begin{equation}
\label{eq:sn-equilibria}
\mathcal E_*^\pm
=
\left(
s_*\pm x_*,
\frac{cs_*}{e}\pm\frac{cx_*}{e+2d_v},
s_*\mp x_*,
\frac{cs_*}{e}\mp\frac{cx_*}{e+2d_v}
\right).
\end{equation}

The equilibria \(\mathcal E_*^+\) and \(\mathcal E_*^-\) are
exchanged by the symmetry
\[
(u_1,v_1,u_2,v_2)
\longmapsto
(u_2,v_2,u_1,v_1).
\]

Near each of these equilibria, after a smooth change of
coordinates and a positive rescaling of time, the reduced
equation on the one-dimensional center manifold has the
saddle-node normal form
\begin{equation}
\label{eq:sn-normal-form}
\dot\eta
=
\frac{c}{e+2d_v}\alpha(b-b_*)
+ \frac{\beta}{2}\eta^2
+ \mathcal O\left(
(b-b_*)^2
+ |b-b_*||\eta|
+ |\eta|^3
\right),
\end{equation}
where
\[
\alpha
=
2(p_1s_*+p_2z_*)
\]
and
\[
\beta
=
-6ap_1q_s^2
+ p_2\left(
-6aq_sq_z+f''(s_*)q_s^2
\right).
\]

Consequently, two nearby equilibria exist on the side of \(b_*\)
determined by
\[
-\frac{
2c\alpha(b-b_*)
}{
(e+2d_v)\beta
}
>0,
\]
whereas no nearby equilibria exist on the opposite side.
\end{thm}

\begin{proof}
The equations
\[
g_1(z,s,b)=0,
\qquad
g_2(z,s,b)=0
\]
correspond to the stationary solutions of
\Cref{eq:2times2xysr_stationary2} with
\[
x\neq0
\qquad\text{and}\qquad
z=x^2>0.
\]

The fold condition is
\[
\det D_{(z,s)}G(z_*,s_*,b_*)=0.
\]
As before, we write a point near \((z_*,s_*)\) as
\[
\begin{pmatrix}
z\\
s
\end{pmatrix}
=
\begin{pmatrix}
z_*\\
s_*
\end{pmatrix}
+ \eta q+W,
\]
where
\[
\langle p,W\rangle=0.
\]

The Lyapunov--Schmidt range equation can be solved uniquely for
\(W\) as a smooth function of \(\eta\) and
\[
\mu=b-b_*.
\]
We define
\[
g(\eta,\mu)
=
\left\langle
p,
G\left(
\begin{pmatrix}
z_*\\
s_*
\end{pmatrix}
+ \eta q+W(\eta,\mu),
b_*+\mu
\right)
\right\rangle.
\]
The reduced equilibrium equation is
\[
g(\eta,\mu)=0.
\]

Since
\[
G(z_*,s_*,b_*)=0
\]
and
\[
J_*q=0,
\]
the constant term and the term linear in \(\eta\) vanish. A
Taylor expansion therefore gives
\begin{equation}
\label{eq:sn-LS-reduced}
g(\eta,\mu)
=
\frac{c}{e+2d_v}\alpha\mu
+ \frac{\beta}{2}\eta^2
+ \mathcal O\left(
\mu^2+|\mu||\eta|+|\eta|^3
\right),
\end{equation}
where
\[
\frac{c}{e+2d_v}\alpha
=
\left\langle
p,
\partial_bG(z_*,s_*,b_*)
\right\rangle
\]
and
\[
\beta
=
\left\langle
p,
D^2_{(z,s)}G(z_*,s_*,b_*)[q,q]
\right\rangle.
\]

The assumptions
\[
\alpha\neq0
\qquad\text{and}\qquad
\beta\neq0
\]
are precisely the standard parameter-transversality and
quadratic-nondegeneracy conditions for a saddle-node
bifurcation.

The center-manifold reduction of the full dynamical system has
the corresponding form
\[
\dot\eta
=
\frac{c}{e+2d_v}\alpha(b-b_*)
+ \frac{\beta}{2}\eta^2
+ \mathcal O\left(
(b-b_*)^2
+ |b-b_*||\eta|
+ |\eta|^3
\right).
\]
This is the standard normal form of a nondegenerate saddle-node
bifurcation. We refer also for example to
\cite{KrupaAmbrosioAzizAlaoui2014} where a related fold phenomenon in two coupled slow-fast systems was analyzed. This completes the proof.
\end{proof}

\textbf{Numerical Application}
For the parameter values in \eqref{eq:param-EDO4v}, the
theoretical formulas yield the following results. At
\(u=\gamma\),
\[
b_1=b_*(\gamma)
=
\frac{121}{12}
\approx10.0833,
\qquad
Q_\gamma<0.
\]
Therefore, the pitchfork bifurcation at \(u=\gamma\) is
supercritical, and the two nonhomogeneous branches exist locally
for \(b>b_1\).

At \(u=1\),
\[
b_3=b_*(1)
=
\frac{319}{12}
\approx26.5833,
\qquad
Q_1>0.
\]
Therefore, the pitchfork bifurcation at \(u=1\) is subcritical,
and the two nonhomogeneous branches exist locally for \(b<b_3\).

For the saddle-node bifurcation, standard numerical methods yield
a solution
\[
(z_*,s_*,b_*)
\]
of
\[
g_1=g_2=h=0,
\]
with the approximate values
\[
z_*\approx0.1578,
\qquad
s_*\approx0.8673,
\qquad
b_*\approx23.14.
\]
These results explain the bifurcation diagram and the dynamical
description presented in Section~6.1.

\section{Conclusion}
In this article, we have introduced and analyzed a modified KS model that reproduces the patterns observed in \textit{ in vitro} BMEC growth experiments and provided mathematical insights into the mechanisms driving this pattern formation. Based on this model, we have also computed distributions of chemoattractant concentrations that are consistent with the time evolution of the experimental growth patterns. New data are currently being generated, and our dynamical-systems approach will be applied in combination with statistical and machine learning methods, with the goal of reducing the dimensionality required for characterization and extracting meaningful biological parameters. This research is part of a broader effort that includes developing biochemical mathematical models for angiogenesis, \textit{ in vitro} models incorporating BMECs, astrocytes, and neurons, as well as three-dimensional models of blood flow, ultimately aiming to create in silico platforms to support the development of new therapeutics for neurodegenerative diseases.

\section*{Acknowledgements}
This work was funded by the National Institute on Aging and the Department of 
Veterans Affairs IK2CX002180, DataPhilantropy, Chan Zuckerberg Initiative and
New Vision Research (to F.M.E.)
\appendix
\section{Implementation of the GMRES method for solving the elliptic equation for the field $v$}
In this study, we were led to the problem of computing a numerical solution to a linear elliptic PDE of the form:
\begin{equation*}
-\nabla u(t)\cdot \nabla v(t)-u\Delta v(t)=\frac{1}{b}\Big( u(t+1)-u(t)-f(u(t))-d_u\Delta u(t) \Big)     
\end{equation*}
with NBC on $\partial \Omega$. Here, $u$ is known (this is the data) and the unknown is $v$. Discretization of the operator $-\nabla u(t)\cdot \nabla v(t)-u\Delta v(t)$ using finite differences leads  to a sparse linear system:
\begin{equation}
\mathbf{A} \mathbf{X} = \mathbf{B},
\end{equation}
where for our choice of space step discretization $\mathbf{A} \in \mathbb{R}^{10^8}$. This is numerically challenging. The \emph{Generalized Minimal Residual (GMRES)} method~\cite{SaadSchultz1986} is an iterative method which does not require to explicit $A$, rather, it only requires to compute $AX$ which belongs to $\mathbb{R}^{10^4}$. This was indeed very efficient to solve our problem. GMRES is an iterative Krylov subspace method, which constructs an approximate solution $\mathbf{X}_k$ in the Krylov subspace
\begin{equation}
\mathcal{K}_k(\mathbf{A}, \mathbf{R}_0) = \text{span}\{\mathbf{R}_0, \mathbf{A} \mathbf{R}_0, \mathbf{A}^2 \mathbf{R}_0, \dots, \mathbf{A}^{k-1} \mathbf{R}_0\},
\end{equation}
where $\mathbf{R}_0 = \mathbf{B} - \mathbf{A} \mathbf{X}_0$ is the initial residual. At each iteration, GMRES computes
\begin{equation}
\mathbf{r}_k = \min_{\mathbf{X} \in \mathbf{X}_0 + \mathcal{K}_k} \| \mathbf{B} - \mathbf{A}\mathbf{X} \|_2.
\end{equation}
This constitutes the main loop of the algorithm that we implemented, which stops when $r_k$ reaches a small threshold value $\epsilon$.
We developed our own C++ code (supplementary material). For our problem, we set $\epsilon=0.1$.
At each step, the algorithm defines a new vector $V_k$, such that the  sequence $V_1,....,V_k$, is orthonormal. This goes as follows, initiate
\[V_1=\frac{r_0}{||r_0||.}\]
Assuming $V_1,...,V_{k-1}$ are known and orthonormal, then we seek
\[\tilde{V}_k=AV_{k-1}-\sum_{j=1}^{k-1}\beta_jV_j.\]
Orthogonality condition
\[(\tilde{V}_k,V_j)=0\]
leads to
\[\beta_j=(AV_{k-1},V_j), j\in\{1,...,k-1\}.\]
Then we normalize 
\[V_k=\frac{\tilde{V}_k}{||\tilde{V}_k||}.\]
This is used to compute $r_k$ as 
\[r_k = \min_{X_k=X_0+\sum_{i=1}^ky_iV_i} \| B - AX_k \|_2. \]
However, we remark that
\[B - AX_k=R_0-\sum_{i=1}^ky_iAV_i\]
and that
\[AV_i-\sum_{j=1}^i(AV_j,V_i)V_j=\tilde{V}_{i+1}=||\tilde{V}_{i+1}||V_{i+1}\]
which gives
\[r_k=\min_{y_1,y_2,...,y_k}\| R_0 - \sum_{i=1}^ky_i(\sum_{j=1}^{i+1}h_{ji}V_j) \|_2\]
where
\[h_{ji}=(AV_j,V_i) 1\leq j\leq i,\,\,h_{i+1,i}=||\tilde{V}_{i+1}||.\]
This rewrites in matricial form
\[r_k=\min_{y\in \R^k}||R_0-(VH)y||\]
where
\[V=(V_1,...,V_{k+1})\]
and
\[H=\begin{pmatrix}
h_{11} & h_{12} & \cdots & h_{1k} \\
h_{21} & h_{22} & \cdots & h_{2k} \\
0      & h_{32} & \cdots & h_{3k} \\
\vdots & \vdots & \ddots & \vdots \\
0      & 0      & \cdots & h_{k,k} \\
0      & 0      & \cdots & h_{k+1,k}
\end{pmatrix}.
\]
Since $V$ is orthonormal and $R_0=V_1r_0$, multiplying by $V$ gives
\[r_k=\min_{y\in \R^k}||r_0e_1-Hy||.\]
Then we further write
\[r_k=\min_{y\in \R^k}||r_0Qe_1-QHy||,\]
with
\[Q=G_kG_{k-1}...G_1\]

and

\[
G_{i}=
\begin{pmatrix}
1 &        &        &        &        &        \\
  & \ddots &        &        &        &        \\
  &        & 1      &        &        &        \\
  &        &        & \cos\theta & -\sin\theta &        \\
  &        &        & \sin\theta & \cos\theta  &        \\
  &        &        &        &        & 1      \\
  &        &        &        &        &        \ddots
\end{pmatrix}
\]
where $\theta$ is chosen so that the  $(i+1,i)$ entry of  $G_i...G_1H$ is zero. It follows that $QH$ is an upper triangular matrix with $k+1$ rows and $k$ columns, with $(k+1)-th$ row consisting entirely of zeros. From this,  we deduce that $r_k=(r_0Qe_1)_{k+1}$, and thus no explicit computation of $y$ is required until the while loop terminates. Finally, we summarize the algorithm below.

\begin{algorithm}[H]
\caption{Generalized Minimal Residual Method (GMRES) to solve the elliptic euqtion}
\label{alg:gmres}
\begin{enumerate}
\item Choose an initial guess $X_0 \in \mathbb{R}^n$ and a tolerance $\varepsilon > 0$.

\item Compute the initial residual
\[
R_0 = B - A X_0, 
\qquad r_0 = \|R_0\|_2.
\]

\item If $r_0<\varepsilon$, stop and set $X = X_0$.

\item Set
\[
V_1 = \frac{R_0}{\|R_0\|_2}
\]

\item Enter the main while loop, while  $r_k >\varepsilon$  do:
\begin{enumerate}
\item Compute
\[
h_{i,k},V_{k+1}, 1\leq i  \leq k+1
\]
as discussed above

\item Apply the previously computed matrix rotations
\[
G_1, G_2, \dots, G_{k-1}
\]
to the vector
\[
(h_{1,k}, h_{2,k}, \dots, h_{k+1,k})^T.
\]

\item Compute a new matrix rotation $G_k$ such that
\[
G_k
\begin{bmatrix}
h_{k,k} \\
h_{k+1,k}
\end{bmatrix}
=
\begin{bmatrix}
\sqrt{h^2_{k,k},
+h^2_{k+1,k}} \\
0
\end{bmatrix}.
\]

\item Apply $G_k$ and compute the residual norm $r_k$.

\item If $r_k \le \varepsilon$, stop.
\end{enumerate}

\item Solve the upper triangular system in $R^k$
\[
R y = r_0Qe_1,
\]

\item Compute the approximation of $v$,
\[
X_k = X_0 + V y.
\]

\end{enumerate}
\end{algorithm}

 %\bibliographystyle{elsarticle-num} 
%  \bibliography{samplebib}

\begin{thebibliography}{10}
\expandafter\ifx\csname url\endcsname\relax
  \def\url#1{\texttt{#1}}\fi
\expandafter\ifx\csname urlprefix\endcsname\relax\def\urlprefix{URL }\fi
\expandafter\ifx\csname href\endcsname\relax
  \def\href#1#2{#2} \def\path#1{#1}\fi

\bibitem{Ste2024}
Global, regional, and national burden of disorders affecting the nervous system, 1990–2021: a systematic analysis for the global burden of disease study 2021, The Lancet Neurology 23~(4) (2024) 344--81.
\newblock \href {https://doi.org/https://doi.org/10.1016/0022-5193(71)90051-8} {\path{doi:https://doi.org/10.1016/0022-5193(71)90051-8}}.

\bibitem{Nich2022}
E.~N. et~al, Estimation of the global prevalence of dementia in 2019 and forecasted prevalence in 2050: an analysis for the global burden of disease study 2019, The Lancet Public Health 7~(2) (2022) e105--e125.
\newblock \href {https://doi.org/https://doi.org/10.1016/S2468-2667(21)00249-8} {\path{doi:https://doi.org/10.1016/S2468-2667(21)00249-8}}.

\bibitem{Win2024}
C.~Windon, F.~M. Elahi, Vascular Cognitive Impairment, Springer International Publishing, Cham, 2024, pp. 399--424.
\newblock \href {https://doi.org/10.1007/978-3-031-49963-0-30} {\path{doi:10.1007/978-3-031-49963-0-30}}.

\bibitem{Owe2024}
C.~D. Owens, C.~B. Pinto, P.~Mukli, R.~Gulej, F.~S. Velez, S.~Detwiler, L.~Olay, J.~R. Hoffmeister, Z.~Szarvas, M.~Muranyi, A.~Peterfi, A.~C. d.~C. Pinaffi-Langley, C.~Adams, J.~Sharps, Z.~Kaposzta, C.~I. Prodan, A.~C. Kirkpatrick, S.~Tarantini, A.~Csiszar, Z.~Ungvari, A.~L. Olson, G.~Li, P.~Balasubramanian, V.~Galvan, A.~Bauer, Z.~A. Smith, T.~W. Dasari, S.~Whitehead, M.~R. Medapti, F.~M. Elahi, A.~Thanou, A.~Yabluchanskiy, Neurovascular coupling, functional connectivity, and cerebrovascular endothelial extracellular vesicles as biomarkers of mild cognitive impairment, Alzheimer's \& Dementia 20~(8) (2024) 5590--5606.
\newblock \href {https://doi.org/https://doi.org/10.1002/alz.14072} {\path{doi:https://doi.org/10.1002/alz.14072}}.

\bibitem{Tor2024}
A.~Torres-Espin, H.~L. Radabaugh, S.~Treiman, S.~S. Fitzsimons, D.~Harvey, A.~Chou, C.~A. Lindbergh, K.~B. Casaletto, L.~Goldberger, A.~M. Staffaroni, P.~Maillard, B.~L. Miller, C.~DeCarli, J.~D. Hinman, A.~R. Ferguson, J.~H. Kramer, F.~M. Elahi, \href{http://dx.doi.org/10.1126/scitranslmed.adk3118}{Sexually dimorphic differences in angiogenesis markers are associated with brain aging trajectories in humans}, Science Translational Medicine 16~(775) (Nov. 2024).
\newblock \href {https://doi.org/10.1126/scitranslmed.adk3118} {\path{doi:10.1126/scitranslmed.adk3118}}.
\newline\urlprefix\url{http://dx.doi.org/10.1126/scitranslmed.adk3118}

\bibitem{Zhu2022}
W.~M. Zhu, A.~Neuhaus, D.~J. Beard, B.~A. Sutherland, G.~C. DeLuca, Neurovascular coupling mechanisms in health and neurovascular uncoupling in alzheimer’s disease, Brain 145~(7) (2022) 2276--2292.
\newblock \href {https://doi.org/10.1093/brain/awac174} {\path{doi:10.1093/brain/awac174}}.

\bibitem{Elahi2023}
F.~M. Elahi, Proteomics analyses of csf and plasma for discovery of early vascular pathologies in neurodegenerative disease, Alzheimer's \& Dementia 19~(S13) (2023) e073997.
\newblock \href {https://doi.org/https://doi.org/10.1002/alz.073997} {\path{doi:https://doi.org/10.1002/alz.073997}}.

\bibitem{Good2007}
A.~M. Goodwin, In vitro assays of angiogenesis for assessment of angiogenic and anti-angiogenic agents, Microvascular Research 74~(2) (2007) 172--183.


\bibitem{Murray}
J.~Murray, Mathematical Biology, Springer, 2010.

\bibitem{Turing}
A.~M. Turing, The chemical basis of morphogenesis, Phil. Trans. R. Soc. B 237 (1952) 37--72.

\bibitem{Amb-2012}
B.~Ambrosio, M.-A. Aziz-Alaoui, Synchronization and control of coupled reaction-diffusion systems of the fitzhugh-nagumo type, Computer and Mathematics with application 64 (2012) 934--943.

\bibitem{Amb-2016}
B.~Ambrosio, M.~A. Aziz-Alaoui, Basin of attraction of solutions with pattern formation in slow{\textendash}fast reaction{\textendash}diffusion systems, Acta Biotheoretica 64~(4) (2016) 311--325.
\newblock \href {https://doi.org/10.1007/s10441-016-9294-z} {\path{doi:10.1007/s10441-016-9294-z}}.

\bibitem{Kel1971}
E.~F. Keller, L.~A. Segel, Traveling bands of chemotactic bacteria: A theoretical analysis, Journal of Theoretical Biology 30~(2) (1971) 235--248.
\newblock \href {https://doi.org/https://doi.org/10.1016/0022-5193(71)90051-8} {\path{doi:https://doi.org/10.1016/0022-5193(71)90051-8}}.

\bibitem{Blanchard2020APOE4}
J.~W. Blanchard, M.~Bula, J.~Davila-Velderrain, L.~A. Akay, L.~Zhu, A.~Frank, M.~B. Victor, J.~M. Bonner, H.~Mathys, Y.-T. Lin, T.~Ko, D.~A. Bennett, M.~Kellis, L.-H. Tsai, \href{https://doi.org/10.1038/s41591-020-0886-4}{Reconstruction of the human blood--brain barrier in vitro reveals a pathogenic mechanism of apoe4 in pericytes}, Nature Medicine 26~(6) (2020) 952--963.
\newblock \href {https://doi.org/10.1038/s41591-020-0886-4} {\path{doi:10.1038/s41591-020-0886-4}}.
\newline\urlprefix\url{https://doi.org/10.1038/s41591-020-0886-4}

\bibitem{Lippmann2012BBB}
E.~S. Lippmann, S.~M. Azarin, J.~E. Kay, R.~A. Nessler, H.~K. Wilson, A.~Al-Ahmad, S.~P. Palecek, E.~V. Shusta, \href{https://doi.org/10.1038/nbt.2247}{Derivation of blood-brain barrier endothelial cells from human pluripotent stem cells}, Nature Biotechnology 30~(8) (2012) 783--791.
\newblock \href {https://doi.org/10.1038/nbt.2247} {\path{doi:10.1038/nbt.2247}}.
\newline\urlprefix\url{https://doi.org/10.1038/nbt.2247}

\bibitem{Stanton2025miBrain}
A.~E. Stanton, A.~Bubnys, E.~Agbas, B.~James, D.~S. Park, A.~Jiang, R.~L. Pinals, L.~Liu, N.~Truong, A.~Loon, C.~Staab, O.~Cerit, H.-L. Wen, D.~Mankus, M.~E. Bisher, A.~K.~R. Lytton-Jean, M.~Kellis, J.~W. Blanchard, R.~Langer, L.-H. Tsai, \href{https://doi.org/10.1073/pnas.2511596122}{Engineered 3d immuno-glial-neurovascular human mibrain model}, Proceedings of the National Academy of Sciences of the United States of America 122~(42) (2025) e2511596122.
\newblock \href {https://doi.org/10.1073/pnas.2511596122} {\path{doi:10.1073/pnas.2511596122}}.
\newline\urlprefix\url{https://doi.org/10.1073/pnas.2511596122}

\bibitem{Mim1996}
M.~Mimura, T.~Tsujikawa, Aggregating pattern dynamics in a chemotaxis model including growth, Physica A: Statistical Mechanics and its Applications 230~(3) (1996) 499--543.
\newblock \href {https://doi.org/https://doi.org/10.1016/0378-4371(96)00051-9} {\path{doi:https://doi.org/10.1016/0378-4371(96)00051-9}}.

\bibitem{Dol2024}
J.~Dolbeault, B.~Perthame, Optimal critical mass in the two dimensional keller–segel model in r2, Comptes Rendus Mathematique 339~(9) (2004) 611--616.


\bibitem{Cor2004}
L.~Corrias, B.~Perthame, H.~Zaag, Global solutions of some chemotaxis and angiogenesis systems in high space dimensions, Milan Journal of Mathematics 72~(1) (2004) 1–28.
\newblock \href {https://doi.org/10.1007/s00032-003-0026-x} {\path{doi:10.1007/s00032-003-0026-x}}.

\bibitem{Kut2012}
K.~Kuto, K.~Osaki, T.~Sakurai, T.~Tsujikawa, Spatial pattern formation in a chemotaxis–diffusion–growth model, Physica D: Nonlinear Phenomena 241~(19) (2012) 1629--1639.
\newblock \href {https://doi.org/https://doi.org/10.1016/j.physd.2012.06.009} {\path{doi:https://doi.org/10.1016/j.physd.2012.06.009}}.

\bibitem{Ama-2026-prep}
H.~Z. A.F.~Garroudji, B.~Ambrosio, F.~Elahi, Mathematical analysis of keller-segel type model for brain micro-endothelial cell development, In preparation (2026).

\bibitem{BookProtter-Weinberger}
M.-H. Protter, H.-F. Weinberger, Maximum Principles in Differential Equations, Springer, 1984.

\bibitem{BookSmoller}
J.~Smoller, Shock Waves and Reaction Diffusion, Springer-Verlag, 1994.

\bibitem{amann1990dynamic}
H.~Amann, Dynamic theory of quasilinear parabolic equations---{II}. {G}lobal existence, Mathematische Zeitschrift 202~(2) (1989) 219--250.
\newblock \href {https://doi.org/10.1007/BF01221575} {\path{doi:10.1007/BF01221575}}.

\bibitem{henry1981geometric}
D.~Henry, Geometric Theory of Semilinear Parabolic Equations, Vol. 840 of Lecture Notes in Mathematics, Springer-Verlag, Berlin, Heidelberg, 1981.


\bibitem{pazy1983semigroups}
A.~Pazy, Semigroups of Linear Operators and Applications to Partial Differential Equations, Vol.~44 of Applied Mathematical Sciences, Springer-Verlag, New York, 1983.


\bibitem{AngiotoolWS}
Angiotool, \url{https://ccrod.cancer.gov/confluence/display/ROB2/Home}, accessed: 2024-12-23.

\bibitem{Zud2011}
E.~Zudaire, L.~Gambardella, C.~Kurcz, S.~Vermeren, A computational tool for quantitative analysis of vascular networks, PLOS ONE 6 (2011) 1--12.


\bibitem{SaadSchultz1986}
Y.~Saad, M.~H. Schultz, Gmres: A generalized minimal residual algorithm for solving nonsymmetric linear systems, SIAM Journal on Scientific and Statistical Computing 7~(3) (1986) 856--869.

\bibitem{BookKuehn}
C.~Kuehn, Multiple Time Scale Dynamics, Springer International Publishing, 2015.
\newblock \href {https://doi.org/10.1007/978-3-319-12316-5} {\path{doi:10.1007/978-3-319-12316-5}}.

\bibitem{kuznetsov2004elements}
Y.~A. Kuznetsov, Elements of Applied Bifurcation Theory, 4th Edition, Vol. 112 of Applied Mathematical Sciences, Springer, Springer Nature Switzerland, 2024.
\newblock \href {https://doi.org/10.1007/978-3-031-22007-4} {\path{doi:10.1007/978-3-031-22007-4}}.

\bibitem{Perko2004}
L.~Perko, Differential Equations and Dynamical Systems, 3rd Edition, Texts in Applied Mathematics, Springer New York, NY, 2020.
\newblock \href {https://doi.org/10.1007/978-1-4613-0003-8} {\path{doi:10.1007/978-1-4613-0003-8}}.

\bibitem{Golubitsky1985}
M.~Golubitsky, D.~G. Schaeffer, The Liapunov---Schmidt Reduction, Springer New York, New York, NY, 1985, pp. 289--336.


\bibitem{Vanderbauwhede2011}
A.~Vanderbauwhede, Lyapunov--Schmidt Method for Dynamical Systems, Springer New York, New York, NY, 2011, pp. 937--952.


\bibitem{KrupaAmbrosioAzizAlaoui2014}
M.~Krupa, B.~Ambrosio, M.~A. Aziz-Alaoui, Weakly coupled two-slow--two-fast systems, folded singularities and mixed-mode oscillations, Nonlinearity 27~(7) (2014) 1555--1574.
\newblock \href {https://doi.org/10.1088/0951-7715/27/7/1555} {\path{doi:10.1088/0951-7715/27/7/1555}}.

\end{thebibliography}



%\begin{thebibliography}{00}

%% For numbered reference style

%   \textit{\LaTeX: a document preparation system},
%   Addison Wesley, Massachusetts,

%\end{thebibliography}

\end{document}